\documentclass{amsart}
\usepackage{SSZ}
\usepackage{categorytheory}
\fullpage

\begin{document}

\title{A combinatorial construction of homology via ACGW categories}
\author[Sarazola]{Maru Sarazola}
\address{School of Mathematics, University of Minnesota, Minneapolis MN, 55455}
\email{maru@umn.edu}
\author[Shapiro]{Brandon T. Shapiro}
\address{Kerchof Hall, 141 Cabell Dr, Charlottesville, VA 22903}
\email{brandonshapiro@virginia.edu}
\author[Zakharevich]{Inna Zakharevich}
\address{Malott Hall, Cornell University, Ithaca, NY 14850}
\email{zakh@math.cornell.edu}
\maketitle

2-Segal spaces arise not only from $S_\dotp$-constructions associated to Waldhausen and (proto) exact categories, but also from $S_\dotp$-constructions associated to certain double-categorical structures. A major step in this direction is due to the work of Bergner--Osorno--Ozornova--Rovelli--Scheimbauer \cite{BOORS:2016}, who propose augmented stable double Segal objects as a natural input for an $S_\dotp$-construction. More recently, another such input has been put forth: ACGW categories. ACGW categories have the advantage that they are combinatorial in nature (as opposed to homotopical or algebraic), and thus have fewer difficult coherence issues to work with. The goal of this paper is to introduce the reader to the key ideas and techniques for working with ACGW categories. To do so, we focus on how homology theory generalizes to ACGW categories, particularly in the central example of finite sets. We show how the ACGW formalism can be used to produce various classical homological algebra results such as the Snake lemma and long exact sequences of relative pairs.

\section{Introduction}

In order to extend the notion of homology beyond algebraic contexts, it is
important to consider exactly what makes homology possible.  The standard
definition of the homology of a chain complex is
\[H_n = \ker d_n/\im d_{n+1}.\] Thus in order to define homology, we must at
least have images, kernels, and quotients.  
In an algebraic setting this is always possible in the context of an
\emph{abelian category}.  To go beyond this it is necessary to analyze exactly
how finite sets fail to be abelian.\footnote{This is somewhat akin to asking
  ``how is a raven like a writing desk?''}  A function between finite sets
always has an image, and quotients between finite sets exist.  If we add a
basepoint to our sets, it is tempting to say that the kernel of a map is the
preimage of the basepoint.  With these definitions it is possible to define
short exact sequences, and thus it may be tempting to try and define chain
complexes of finite sets in the following manner.

\begin{definition}
  A \emph{chain complex of finite sets} is a sequence $C_n$ of finite pointed
  sets, together with functions $d_n: C_n \rto C_{n-1}$, satisfying the
  condition that $d_nd_{n+1}$ is the constant map at the basepoint for all $n$.
  We write such a chain complex as $(C_*,d^C)$.

  A \emph{naive map of chain complexes of finite sets}
  $(C_*,d^C) \rto (D_*,d^D)$ is a sequence of functions $f_n: C_n \rto D_n$ such
  that $d^D_nf_n = f_{n-1}d^C_n$ for all $n$.

  The category of chain complexes of finite sets and naive maps between them is
  denoted $n\Ch_\FinSet$.
\end{definition}

This intuitive definition 
has several
downsides, all coming from the fact that the \emph{kernel} of a surjective map
does not satisfy the kinds of properties that we would like.  For example, it
should intuitively be the case that if there is a surjective map $f:A \rfib B$
then $|A| = |B| + |\ker f|$.  If $f$ is injective away from the basepoint this
is true, but otherwise it is not.  



In a forthcoming paper, Sarazola--Shapiro suggest defining ``surjective'' maps
somewhat differently:
\begin{definition}
  An \emph{admissible monomorphism} between chain complexes of finite sets is a
  naive map of chain complexes of finite sets which is levelwise injective.
  
  An \emph{admissible epimorphism} between chain complexes of finite sets is a
  naive map of chain complexes of finite sets which is, at each level, bijective away from the basepoint (and consequentially surjective).

  The category of chain complexes of finite sets and morphisms which can be
  factored as an admissible epimorphism followed by an admissible monomorphism
  is denoted
  $a\Ch_\FinSet$.  The morphisms of this category are called \emph{admissible
    morphisms}. 
\end{definition}
With this definition, they construct a $K$-theory spectrum for $n\Ch_\FinSet$ using only the short exact sequences from $K(a\Ch_\FinSet)$ and show that it agrees with $K(\FinSet)$.  The key observation here is that it is admissible epimorphisms rather than arbitrary surjective maps between chain complexes that play a role alongside monomorphisms in the construction of $K$-theory.


However, in order to prove this theorem it is necessary that arbitrary
surjective maps are allowed to appear as differentials inside the chain
complex.  This means that while it is possible to define the kernel of a map
\emph{between} chain complexes, it is not possible to define the kernel of a
\emph{differential}.   In order to make kernels possible, we must also restrict
the differentials.

\begin{definition}
  An \emph{injective chain complex of finite sets} is a chain complex of finite
  sets in which the differentials are injective away from the basepoint.

  The full subcategory of injective chain complexes inside $a\Ch_{\FinSet}$ is
  denoted $i\Ch_{\FinSet}$.
\end{definition}

It turns out that many of the proofs in homological algebra rely on a simple
duality structure: that in an abelian category there is, for every object $A$, a
natural bijection between injections into $A$ and surjections out of $A$.  Thus
any result that relies only on this observation should be generalizable into a
context where this is the case.  This is the notion of an \emph{ACGW category},
originally defined in \cite{CZ-cgw} and further developed in \cite{SS-cgw}.


\begin{theorem}\label{thm:main}
  Let $\bbA$ be an ACGW category, and write $\CCh_\bbA$ for the ACGW category of
  chain complexes over $\bbA$.  Then for all $n$ there exists a functor
  $H_n: i\Ch_\bbA \rto \bbA^\flat$ (where $\bbA^\flat$ is a certain ordinary
  category associated to $\bbA$); see Proposition~\ref{prop:homologyonmaps} and   Theorem~\ref{thm:functorialhomology}.  Moreover, the functors $H_n$ satisfy
   standard theorems of homology theories: the Snake Lemma
  (Theorem~\ref{thm:snakelemma}), and the Long Exact Sequence of a relative pair
  (Theorem~\ref{thm:lesforhomology}).
\end{theorem}

Moreover, the definition of homology yields a natural definition of
quasi-isomorphism (see Section~\ref{subsection:quasiiso}).

When $\bbA$ is the ACGW category associated to an abelian category $\A$,
$\bbA^\flat = \A$, so the theorem recovers the standard homological algebra results. The ACGW category $\FFinSet$ has its associated category of chain complexes equivalent to $i\Ch_\FinSet$, and $\FFinSet^\flat$ is equivalent to the category of
finite pointed sets and functions which are injective away from the basepoint.

The goal of this paper is to prove the above theorem by largely focusing on the
case of finite sets.  Although all of our proofs will hold in any ACGW category, all of the intuition and illustrations will use finite sets. For instance, the Snake Lemma in the case of $\FFinSet$ can be visualized using the following pair of pictures.
\[
\raisebox{-.15\height}{\includegraphics[height=4cm]{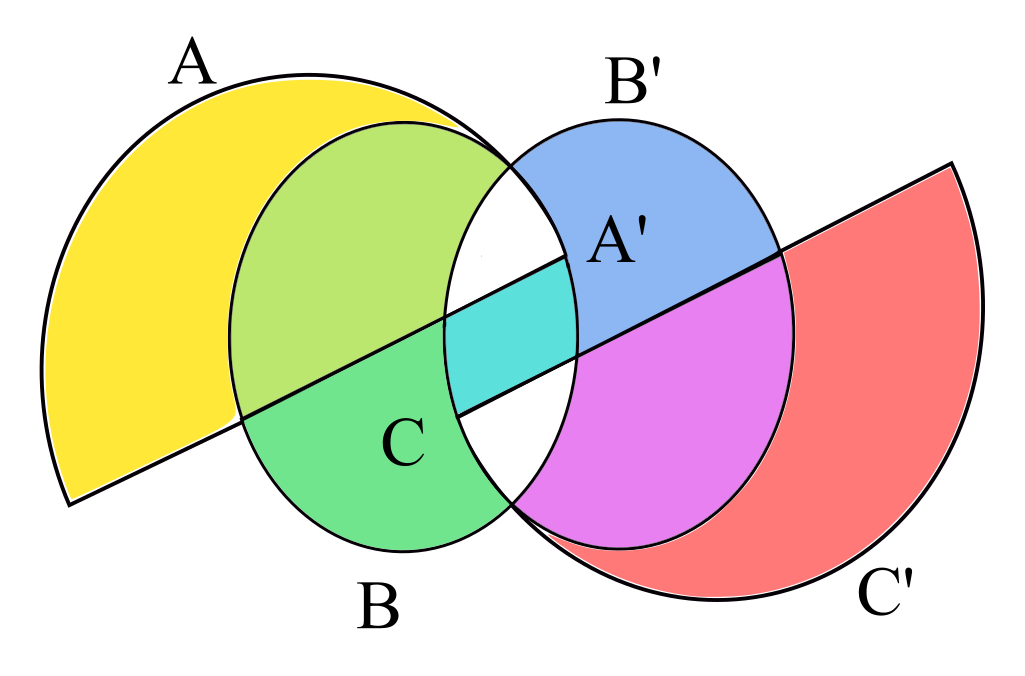}}
\qquad\quad
\includegraphics[height=3cm]{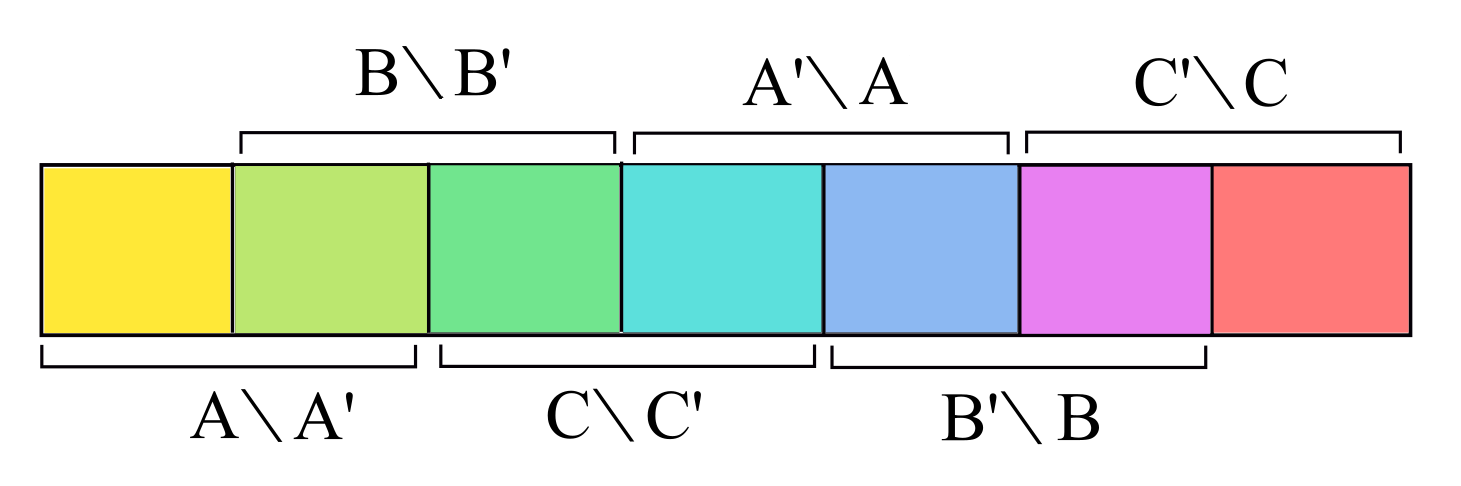}
\]

\subsection*{Organization} This paper is organized as follows.  In Section
\ref{section:ACGWcats} we explain the working features of an ACGW category, and
explain the ACGW structure associated to finite sets.  In Section
\ref{section:homologydefn} we define chain complexes over an ACGW category and construct
the homology functors $H_i$, and explain them in detail in the special case of finite
sets.  In Section \ref{section:results} we prove all the components of Theorem \ref{thm:main}.

\subsection*{Acknowledgements} The authors would like to thank the Banff
International Research Station for hosting the workshop ``Higher Segal Spaces and their Applications to Algebraic $K$-Theory, Hall Algebras, and Combinatorics''.
We would also like to thank Shruthi Sridhar-Shapiro for making the illustrations in this
paper for us.  Zakharevich was supported in part by a Simons Foundation
Fellowship, and by NSF CAREER DMS-1846767.

\section{ACGW categories}\label{section:ACGWcats}

In this section we describe ACGW categories and their properties.  Their aim is
to axiomatize several features common to both algebraic categories (where we
have access to a notion of (co)kernels) and to combinatorial categories (where
we have a notion of complements). As a consequence, their definition can be
quite technical, especially for readers who are not already acquainted with the
language of double categories. We do not give a complete, formal
definition of ACGW categories. Instead, we present an intuitive approach that
only highlights the essential features, with an emphasis on the examples of
finite sets and finite pointed sets. A thorough reader may find the definition in \cite[Definition
5.6]{CZ-cgw}; generalizations for the interested reader can be found in
\cite[Definition 2.5]{CZ-cgw} and \cite[Definitions 2.4, 4.1]{SS-cgw}.

A double category is a collection of objects with \emph{two} different
category structures between them, as well as data about how these two different
category structures relate to one another.  


\begin{definition}
    A \emph{double category} 
    $\bbA$ consists of objects $A, B, A', B',\dots$, horizontal morphisms
    $A\mrto B$, vertical morphisms $A\erto A'$, and \emph{pseudo-commutative squares}
    \begin{diagram}
    { A & B \\
      A' & B' \\};
      \mto{2-1}{2-2} \eto{1-2}{2-2} \mto{1-1}{1-2} \eto{1-1}{2-1} \labar{1-1}{2-2}{\circlearrowleft}
    \end{diagram}
together with associative and unital compositions for horizontal
morphisms, vertical morphisms, and squares. 

While in general a pseudo-commutative\footnote{In the literature these
  are more commonly called ``$2$-cells'' or simply ``squares''.} square is not necessarily uniquely
determined by the morphisms on its boundary, in all of our
examples each diagram of the above shape either \emph{is} a square in the double category structure, or it is \emph{not} a square in the double category structure.  
\end{definition}

\begin{notation}
Given objects $A,B$ in a double category $\bbA$, we denote by $\Hor(A,B)$ the category whose objects are the horizontal morphisms $A\mrto B$ in $\bbA$, and whose morphisms are given by squares in $\bbA$. Similarly, we use $\Ver(A,B)$ for the category of vertical morphisms $A\erto B$ and squares in $\bbA$.
\end{notation}

\begin{remark}
  We will use blackboard bold for double categories and ACGW categories, and
  ordinary bold or caligraphic font for ordinary categories.  Thus $\FinSet$ and
  $\C$ are ordinary categories, but $\FFinSet$ and $\bbA$ are double categories.
\end{remark}

\begin{definition}\label{defn:ACGW}
An \emph{ACGW category} is a double category  $\bbA$ satisfying the
following extra conditions:
\begin{itemize}
    \item There is an object $\varnothing\in\bbA$ which is initial both
      for the horizontal and the vertical morphisms, and both types of morphisms
      are monic.
    \item For every object $A\in \bbA$ there is a pair of inverse bijections
      \[c: \bigcup_{B\in \bbA} \Hor(B,A) \rlto \bigcup_{B\in \bbA} \Ver(B,A)
        \ :\!k.\]
      The function $c$ is the ``cokernel,'' or ``complement,'' and the function
      $k$ is the ``kernel.''  For a horizontal morphism $A \mrto B$ we write $B
       \fforwardslash  A \erto B$ for its cokernel, and for a vertical morphism $A \erto B$ we
      write $B \bbackslash A \mrto B$ for its kernel.
    \item The pair above is natural in the following sense.  For every
      vertical morphism $B \erto D$, the
      bijection $c$ extends to a bijection
      \begin{equation} \label{eq:squareduality}
        c:\left\{\csqinline{A}{B}{C}{D}{}{}{}{}\right\} \rto \left\{\mathrm{pullbacks}\ 
          \begin{inline-diagram}
            {X & B \\ Y & D \\};
            \eto{1-1}{1-2} \eto{2-1}{2-2} \eto{1-1}{2-1} \eto{1-2}{2-2}
          \end{inline-diagram} \right\}
      \end{equation}
      An analogous statement holds for $k$ for any horizontal morphism $C \mrto
      D$.
      
      As a consequence of this, given any diagram
      \[A \mrto C \elto B\]
      there exists a unique (up to unique isomorphism) ``pullback
      pseudo-commutative square'' 
      \[\csq{A \oslash_C B}{B}{A}{C}{}{}{}{}.\]

\end{itemize}
\end{definition}

Again, we emphasize that this definition is only an intuitive one. Formally, the assignments $c$ and $k$ are part of the data of an ACGW category, and we need to be slightly more careful when describing their codomain categories, something we glossed over here. These functors are also required to be equivalences of categories. Finally, there are a host of additional axioms one can ask for, that allow us to encode notions like pushouts/pullbacks, image factorizations, weak equivalences, and so on; these give rise to more specialized variants of ACGW categories that we will not discuss here, but that the reader can find in \cite{CZ-cgw, SS-cgw}.   

Since our goal is to present several classical homological algebra results using the language of ACGW categories, we introduce a combinatorial example that will serve as one source of intuition.

\begin{example}
  The ACGW category $\FFinSet$ has
  \begin{description}
  \item[objects] finite sets,
  \item[horizontal morphisms] injections,
  \item[vertical morphisms] injections, and
  \item[pseudo-commutative squares] squares which are pullbacks when considered
    as a diagram in $\FinSet$.
  \end{description}
  
  The object $\varnothing$ is the empty set.  Both $c$ and $k$ will be the same:
  for any injection $A \rto B$ we take the inclusion $B \backslash A \rto B$
  (i.e. the inclusion of the complement).  Then the duality between
  pseudo-commutative squares and pullback squares states that, if
  $B, C \subseteq D$ then
  \[B \cap C = B \backslash \left((D \backslash C) \cap B\right).\]
  The construction $\oslash$ is simply the pullback in the ordinary category $\FinSet$. 

  
       All of the relevant squares in this double category can be visualized using pictures such as those below. For each picture on the left, we include on the right the various types of diagrams in the ACGW category of sets it could correspond to.
    \begin{equation}
    \raisebox{-.5\height}{\includegraphics[width=4cm]{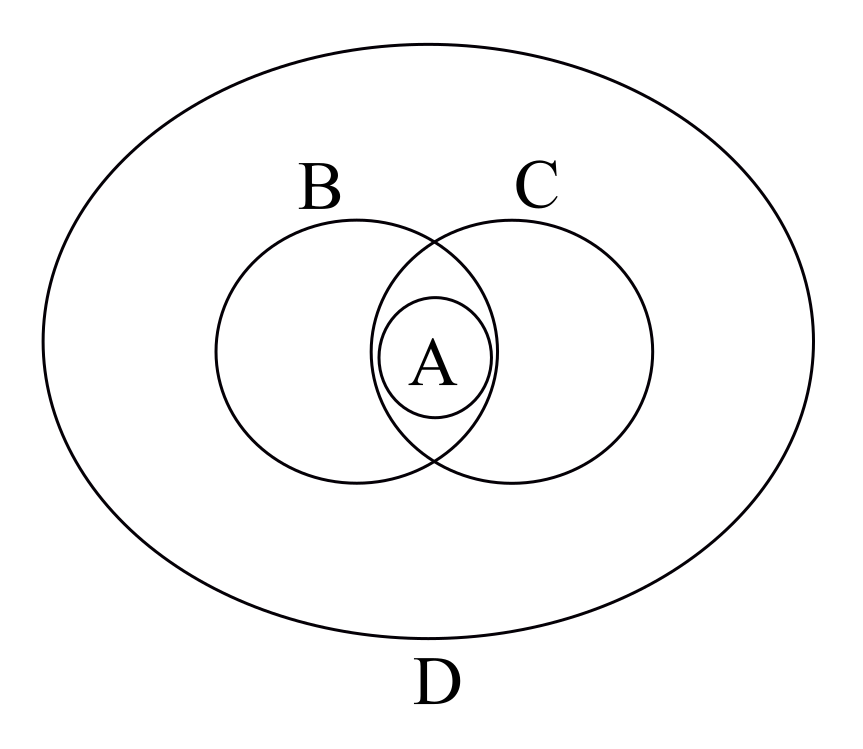}}
    \qquad\qquad
    \begin{tikzpicture}[baseline]%
      \matrix (m) [matrix of math nodes,row sep=2.6em, column sep=2.8em]
      { A & B \\ C & D \\};
      \mto{1-1}{1-2} \mto{1-2}{2-2} \mto{1-1}{2-1} \mto{2-1}{2-2}
    \end{tikzpicture}
    \qquad
    \begin{tikzpicture}[baseline]%
      \matrix (m) [matrix of math nodes,row sep=2.6em, column sep=2.8em]
      { A & B \\ C & D \\};
      \eto{1-1}{1-2} \eto{1-2}{2-2} \eto{1-1}{2-1} \eto{2-1}{2-2}
    \end{tikzpicture}
  \end{equation}
    In a commuting square of inclusions (of either arrow type), the set $A$ includes into the intersection of $B$ and $C$ but may not cover all of $B \cap C$.
    \begin{equation}\label{eqn.setpullback}
    \raisebox{-.5\height}{\includegraphics[width=4cm]{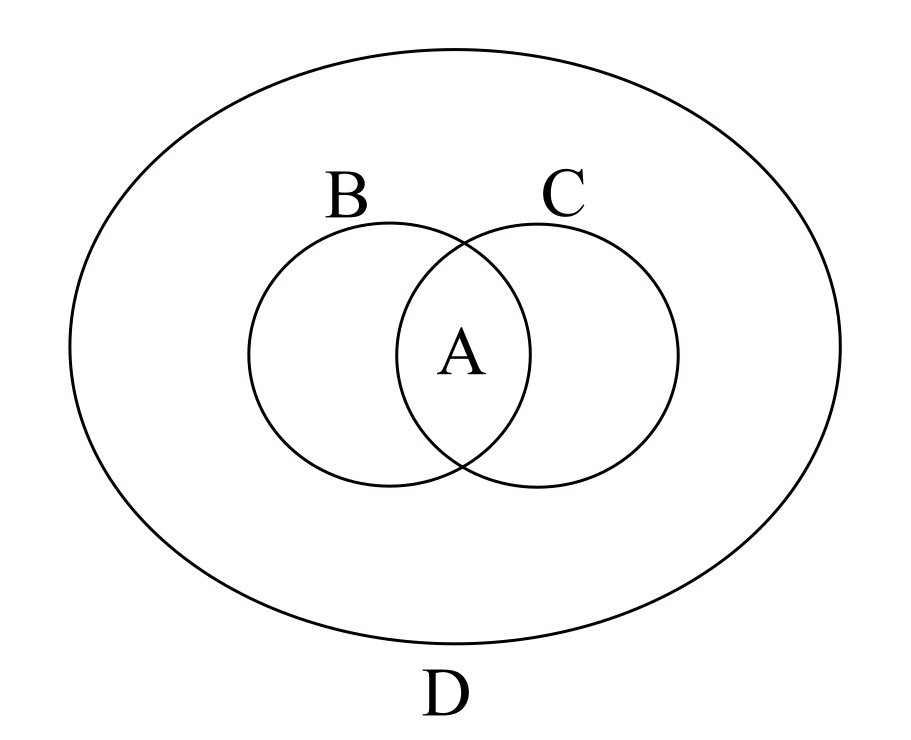}}
    \qquad\qquad 
    \csq{A}{B}{C}{D}{}{}{}{}
    \qquad
    \mpbsq{A}{B}{C}{D}{}{}{}{}
    \qquad
    \epbsq{A}{B}{C}{D}{}{}{}{}
    \end{equation}
    In a cartesian square (be it of horizontal morphisms, vertical morphisms, or a mix) we moreover have $A = B \cap C$.

    For the pseudo-commutative square in \eqref{eqn.setpullback}, its  cartesian complement square in the horizontal direction models the intersection of $B$ and $D \backslash C$ in $D$.
    \begin{equation}
    \raisebox{-.5\height}{\includegraphics[width=4cm]{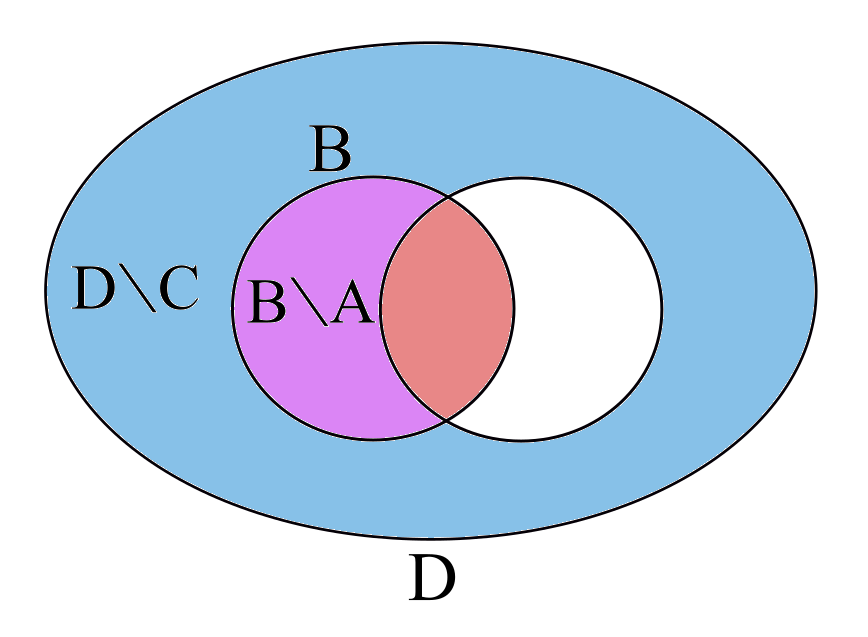}}
    \qquad\qquad\qquad
    \epbsq{B \backslash A}{B}{D \backslash C}{D}{}{}{}{}
    \end{equation}
    In this picture $B$ is shaded red, $D \backslash C$ shaded blue, and their intersection $B \backslash A$ shaded purple.
\end{example}

\begin{remark}
    The primary references \cite{CZ-cgw,SS-cgw} focus on the example of finite sets, as they are concerned with $K$-theory, which requires finiteness.  By contrast, nothing in this paper requires finiteness, so all results about $\FFinSet$ are equally true in the ACGW category $\mathbb{S}\mathbf{et}$; in the interest of being faithful to our sources we stay with the finite set example.
\end{remark}

\begin{example}
    Any abelian category $\mathcal A$ gives rise to an ACGW category in a natural way. The double category $\mathbb{A}$ to consider has as objects the objects of $\mathcal{A}$, as horizontal morphisms the monomorphisms, and as vertical morphisms the opposite of the epimorphisms; that is, $\mathbb{A}$ has a vertical morphism $A\erto B$ for each epimorphism $B\twoheadrightarrow A$ in $\mathcal{A}$. Squares are given by the commutative squares in $\mathcal{A}$; that is, the diagram in $\mathbb{A}$ as below left is a square in the double category precisely when the corresponding diagram in $\mathcal{A}$ as below right is commutative. 
    \begin{diagram}
        {A & B & \qquad & A & B\\
        C & D & \qquad & C & D\\};
        \mto{1-1}{1-2} \mto{2-1}{2-2} \mto{1-4}{1-5} \mto{2-4}{2-5} \eto{1-1}{2-1} \eto{1-2}{2-2} \diagArrow{->>}{2-4}{1-4} \diagArrow{->>}{2-5}{1-5}
    \end{diagram}
   The object $\varnothing$ is the zero object, and the functors $c$ and $k$
   consist of taking cokernels and kernels, respectively. 
\end{example}

\begin{example}
  The ACGW category $\FFinSet_\ast$ has
  \begin{description}
  \item[objects] finite pointed sets,
  \item[horizontal morphisms] injections,
  \item[vertical morphisms] surjections which are injective away from the
    preimage of the basepoint, and
  \item[pseudo-commutative squares] squares which are commutative when
    considered in $\FinSet_\ast$ (analogously to the abelian case).
  \end{description}
  The vertical morphisms are precisely the cokernels of inclusions, which send
  the image of an inclusion to the basepoint. Dually, the kernel of such a
  surjection is the inclusion of the preimage of the basepoint.
  
  The ACGW category $\FFinSet_\ast$ is isomorphic to the ACGW category
  $\FFinSet$ via the double functor $\FFinSet \rto \FFinSet_\ast$ taking a set
  $A$ to $A_+$ (adding a disjoint basepoint).  On horizontal morphisms this
  simply takes a function $f$ to $f_+$.  On vertical morphisms it takes an
  injection $g:I \rto J$ to the surjection $\widetilde g: J_+ \rto I_+$ given by
  $\widetilde g(j) = g^{-1}(j)$ if $j$ is in the image of $g$, and $g(j) = *$
  otherwise.  This isomorphism is what motivates much of the intuition in the
  unpointed construction.
\end{example}

There is a natural way to assign an ordinary category to any ACGW category.

\begin{definition}
  Given an ACGW category $\bbA$, there is an ordinary category which is a
  \emph{flattening} of it, denoted $\bbA^\flat$.  This category has
  \begin{description}
  \item[objects] the objects of $\bbA$,
  \item[morphisms] from $A$ to $B$ are (isomorphism classes of) diagrams
    \[A \elto \bar A \mrto B,\]
  \item[composition] of morphisms $A \elto \bar A \mrto B$ with $B \elto \bar B \mrto C$
    given by the composition
    \[A \elto \bar A \oslash_B \bar B \mrto C.\]
    In other words, we take a category of spans with ``legs in different
    directions'' and use the existence of $\oslash$ to compose.
  \end{description}

  Every morphism $A \elto \bar A \mrto B$ in $\bbA^\flat$ has a \emph{kernel}
  (given by the kernel of $\bar A \erto A$) and an \emph{image}, given by
  $\bar A$.  A \emph{zero morphism}, denoted $0$, is the unique morphism where
  $\bar A = \varnothing$; there is a zero morphism between any two objects, and
  the composition of any morphism with a zero morphism is a zero morphism.
\end{definition}

If $\bbA$ is the ACGW category associated to an abelian category $\A$, then $\bbA^\flat \cong \A$---this is essentially the epi-mono factorization of morphisms in $\A$.  


There is an important difference between the ACGW category associated to an
abelian category and the ACGW category $\FFinSet$.  Morphisms in $\FinSet$ have
associated epi-mono factorizations, so it may be tempting to try and define an
ACGW structure where the vertical morphisms correspond to opposite surjections.
However, a problem arises: there is no longer a bijection between surjections
out of $A$ and injections into $A$---there are far more surjections---unlike in
the abelian case, where the bijection is completely natural.  

If we want a class
of surjections out of a finite pointed\footnote{Without a basepoint, this is not
  possible at all} set $A$ to be in bijection with injections into $A$, we must
take those surjections which are injective away from the basepoint.  The
category of these is exactly the category of injections between finite sets,
which explains the ACGW structure we chose for $\FFinSet_\ast$. From this we immediately get that $\FFinSet_\ast^\flat$ is the category of finite
pointed sets and morphisms which are injective away from the basepoint. $\FFinSet^\flat$ is the equivalent category of finite sets and (isomorphism classes of) spans of monomorphisms $A \elto \bar A \mrto B$, where the corresponding map of pointed sets $A_+ \to B_+$ sends $A \backslash \bar A$ to the basepoint.

\section{Homology in an ACGW category}\label{section:homologydefn}

Our goal is to use the duality structure present in the definition of an
ACGW category in order to define chain complexes and homology.  

\begin{definition}
  Let $\bbA$ be an ACGW category.  A \emph{chain complex in $\bbA$} is a
  sequence of morphisms $d_i: A_i \rto A_{i-1}$ in $\bbA^\flat$ such that
  $d_id_{i+1}=0$ for all $i$.
\end{definition}

Unwinding this definition, a chain complex consists of a diagram in $\bbA$ of the form
\begin{diagram}
    {\cdots X_{i+1} & \bar X_{i+1} & X_i & \bar X_i & X_{i-1} \cdots\\};
    \eto{1-2}{1-1} \mto{1-2}{1-3} \eto{1-4}{1-3} \mto{1-4}{1-5}
\end{diagram}
along with squares in $\bbA$ of the form below.
\[
\csq{\varnothing}{\bar X_i}{\bar X_{i+1}}{X_i}{}{}{}{}
\]

\begin{example}
  Suppose that $\bbA$ is the ACGW category associated to an abelian category
  $\A$.  Then $\bbA^\flat = \A$.  Moreover, the zero morphism in $\bbA^\flat$ is
   the zero morphism in $\A$, so the definition is exactly the usual
  definition of a chain complex in $\A$.
\end{example}

\begin{example}
  In the case of $\FFinSet_\ast$, a chain complex is an injective chain complex of finite sets, where similarly the chain condition asserts that the composition of two differentials is constant at the basepoint.  
\end{example}

\begin{example}
  A chain complex in $\FFinSet$ has the form of a sequence of spans of inclusions.
  \begin{diagram}
    {\cdots X_{i+1} & \bar X_{i+1} & X_i & \bar X_i & X_{i-1} \cdots\\};
    \eto{1-2}{1-1} \mto{1-2}{1-3} \eto{1-4}{1-3} \mto{1-4}{1-5}
  \end{diagram}
  While by the equivalence between $\FFinSet^\flat$ and $\FFinSet_\ast^\flat$ this is the same information as a chain complex in $\FFinSet_\ast$, we can reason about it using different techniques. 
  
  For instance, consider what it means for $d_nd_{n+1} = 0$ to hold. Suppose that we have two morphisms in
  $\FFinSet^\flat$:
  \[A \elto \bar A \mrto B \elto \bar B \mrto C.\]
  Up to isomorphism, this is the data of three sets $A$,$B$,$C$, such that $A
  \cap B = \bar A$ and $B \cap C = \bar B$.  The composition of these two
  morphisms is then $A \elto \bar A \cap \bar B \mrto C$.  This will be the zero
  morphism exactly when $\bar A \cap \bar B = \varnothing$---in other words,
  when the transition sets $\bar A$ and $\bar B$ do not intersect in $B$.

  Then, a chain complex consists of a sequence of sets $X_i$ equipped with intersections $\bar X_i = X_i \cap X_{i-1}$ such that $\bar X_{i+1}$ and $\bar X_i$ are disjoint in $X_i$. We can therefore visualize a chain complex as in the picture below (which coincidentally looks a bit like a real chain).
  \begin{equation}
  \cdots\raisebox{-.4\height}{\includegraphics[height=3cm]{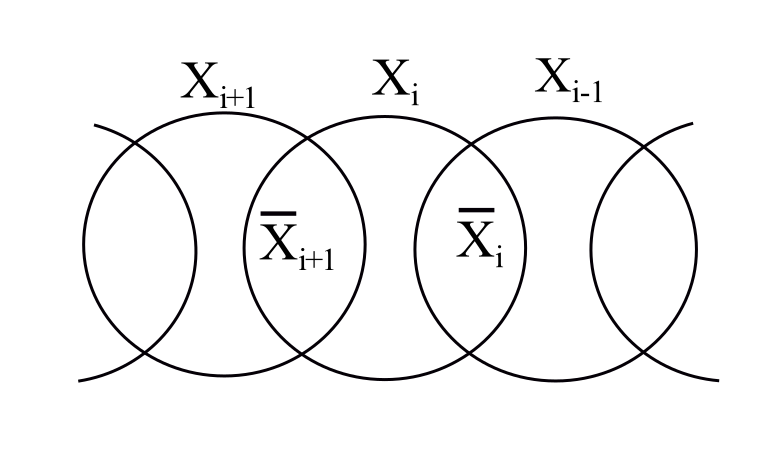}}\cdots
  \end{equation}

  Notice that this structure gives a natural decomposition of each $X_i$ into
  three subsets: $\bar X_{i+1}$, $\bar X_i$, and $X_i \backslash (\bar X_{i+1} \cup \bar X_i)$.
  This will be relevant in Definition \ref{defn:homology} when we define the homology of a chain complex.
\end{example}

\begin{remark}\label{freemodule}
  We have now discussed how a chain complex of sets
  \begin{diagram}
    {\cdots X_{i+1} & \bar X_{i+1} & X_i & \bar X_i & X_{i-1} \cdots\\};
    \eto{1-2}{1-1} \mto{1-2}{1-3} \eto{1-4}{1-3} \mto{1-4}{1-5}
  \end{diagram}
  can be interpreted combinatorially as a sequence of sets with specified intersections, or (less combinatorially, almost algebraically) as a sequence of maps between pointed sets. In fact there is also a fully algebraic interpretation of these complexes: for any ring $R$, chain complexes of sets describe the chain complexes of free $R$-modules whose differentials are composed of only projections and coproduct injections. In particular, given a chain complex of sets as above we get a chain complex of $R$-modules
  \[
  \cdots \to R^{X_{i+1}} \to R^{X_i} \to R^{X_{i-1}} \to \cdots
  \]
  where the differentials factor as the composite $R^{X_i} \to R^{\bar X_i} \to R^{X_{i-1}}$ which projects from $R^{X_i}$ to the components in the image of $\bar X_i$ and inserts 0s for each element of $X_{i-1}$ outside the image of $\bar X_i$. 
  
  This interpretation, induced by a structure-preserving double functor $\FFinSet \to R\text{-}\mathbb{M}\mathbf{od}$ sending $A$ to $R^A$, respects homology and therefore completely determines the homological theory of finite sets. While we won't discuss it further, it encourages the interpretation of finite sets as ``modules over the field with one element'' suggested by Jacques Tits.
\end{remark}

In any pseudo-commutative square as below left, its complementary square in the vertical direction has the form below right,
\begin{equation}
\csq{\varnothing}{B}{A}{C}{}{}{}{}\qquad\qquad\qquad
\mpbsq{A}{C\bbackslash B}{A}{C}{}{}{}{}
\end{equation}
which in particular includes a horizontal morphism from $A$ to $C \bbackslash B = k(B \erto C)$.  In the case of
finite sets this says that if $A$ and $B$ are disjoint subsets of $C$, then there is a natural inclusion of $A$ into $C \backslash B$.  (In the case of an abelian category, this expresses the statement that if $gf = 0$ then $f$ factors through $\ker g$.) There is similarly a vertical morphism from $B$ to $C \fforwardslash A = c(A \mrto C)$.

\begin{definition}\label{defn:homology}
  Let $\bbA$ be an ACGW category.  The \emph{$i$-th homology} $H_i(X_\dotp)$ of a chain complex $X_\dotp$ in $\bbA$ 
  \begin{diagram}
    {\cdots X_{i+1} & \bar X_{i+1} & X_i & \bar X_i & X_{i-1} \cdots\\};
    \eto{1-2}{1-1} \mto{1-2}{1-3} \eto{1-4}{1-3} \mto{1-4}{1-5}
  \end{diagram}  
  is defined by taking complements in both directions (in either order) of the squares defining the chain condition, as in \eqref{eqn.homology}. 
  \begin{diagram-numbered}{eqn.homology}
      {& & H_i(X_\dotp) \\ & X_i \bbackslash \bar X_i & & X_i  \fforwardslash  \bar X_{i+1} \\ \bar X_{i+1} & & X_i & & \bar X_i \\ & \bar X_{i+1} & & \bar X_i \\ & & \varnothing \\};
      \comm{2-2}{2-4} \comm{4-2}{4-4}
      \mto{5-3}{4-4} \eto{5-3}{4-2}
      \eq{4-2}{3-1} \eq{4-4}{3-5}
      \mto{4-2}{3-3} \eto{4-4}{3-3}
      \mto{2-2}{3-3} \eto{2-4}{3-3}
      \mto{3-1}{2-2} \eto{3-5}{2-4}
      \mto{1-3}{2-4} \eto{1-3}{2-2}
  \end{diagram-numbered}
\end{definition}

\begin{example}
   Let us compare this definition to the usual notion of homology for an abelian
   category in a bit more detail. If we consider a chain complex (with the epi-mono factorization of each differential) 
    \begin{equation}
    \begin{tikzcd}
        \cdots X_{i+1}\ar[rr,"d_{i+1}"]\ar[rd,"p_{i+1}"',->>] & & X_i\ar[rr,"d_{i}"]\ar[rd,"p_{i}"',->>] & & X_{i-1}\cdots\\
        & \bar X_{i+1}\ar[ur,>->,"j_{i+1}"'] & & \bar X_{i}\ar[ur,>->,"j_{i}"'] & & 
    \end{tikzcd}
    \end{equation} 
    then the $i$-th homology is given by $H_i(X_\dotp)=\frac{\ker d_i}{\im d_{i+1}}$. Note that $\im d_{i+1}=\bar X_{i+1}$, and that $\ker d_i=\ker (j_i p_i)=\ker p_i$; hence we can rewrite the homology as $H_i(X_\dotp)=\frac{\ker p_i}{\bar X_{i+1}}$. This quotient can be captured in the diagram \eqref{eqn.abhomology}
    \begin{equation}\label{eqn.abhomology}
    \begin{tikzcd}
        \bar X_{i+1}\dar[equal]\rar[>->] & \ker p_i\dar[>->]\rar[->>] & \frac{\ker p_i}{\bar X_{i+1}}=H_i(X_\dotp)\\
        \bar X_{i+1}\ar[dr,phantom,"\circlearrowleft"]\rar[>->] \dar[->>] & X_i\dar["p_i", ->>] & \\
        0\rar & \bar X_i & 
    \end{tikzcd}
    \end{equation}
    which is precisely the left half of the diagram defining homology. We start with the commutative square on the bottom given by the chain condition, take the kernels of the vertical epis, and then the cokernel of the induced inclusion $\bar X_{i+1}\hookrightarrow \ker p_i$. The other half of the grid diagram is redundant, but illustrates the alternative definition of homology as the kernel of the induced map $X_i/\bar X_{i+1} \to \bar X_i$.
\end{example}

\begin{example}
    Consider now the case of $\FFinSet$. Translating the diagram above for the ACGW category of sets, we obtain the diagram 
    \begin{diagram-numbered}{}
    {\bar X_{i+1} & X_i\backslash \bar X_i & \left(X_i\backslash\bar X_i\right)\backslash \bar X_{i+1}  \\
    \bar X_{i+1} & X_i \\
    \emptyset & \bar X_i\\};
    \mto{1-1}{1-2} \eq{1-1}{2-1} \eto{3-1}{2-1} \eto{1-3}{1-2} \mto{1-2}{2-2} \eto{3-2}{2-2} \mto{3-1}{3-2} \mto{2-1}{2-2} 
    \path[font=\scriptsize] (m-3-1) edge[-,white] node[pos=0.05]
  {\textcolor{black}{$\urcorner$}}  (m-2-2);
    \end{diagram-numbered}
    and we see that $H_i(X_\dotp) = \left(X_i\backslash\bar X_i\right)\backslash \bar X_{i+1}=X_i\backslash \left( \bar X_{i+1}\cup\bar X_i\right)$. In this case the symmetry in the definition is even clearer as we also have $\left(X_i \backslash \bar X_{i+1}\right) \backslash \bar X_i = X_i \backslash \left(\bar X_{i+1} \cup \bar X_i\right)$.

  In other words, the $i$-th homology measures the elements of $X_i$ that are not reached
  by either differential. In pictures, we can see the homology in a chain complex of sets as in \eqref{eqn.sethomology}.
  \begin{equation}\label{eqn.sethomology}
  \raisebox{-.45\height}{\includegraphics[height=4cm]{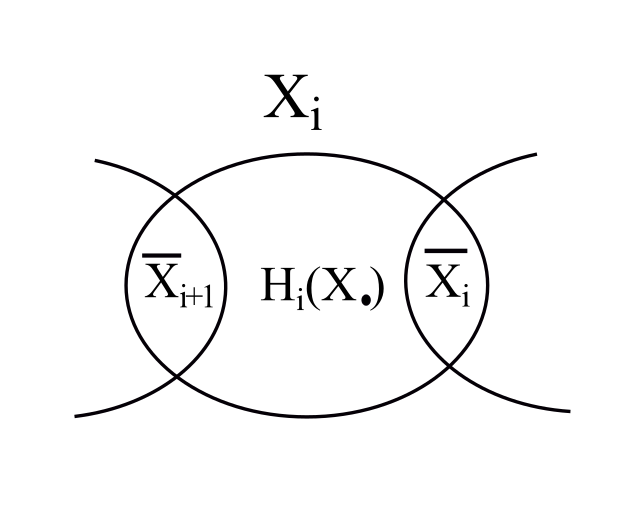}} \qquad\quad\begin{inline-diagram}
    {\cdots X_{i+1} & \bar X_{i+1} & X_i & \bar X_i & X_{i-1} \cdots\\};
    \eto{1-2}{1-1} \mto{1-2}{1-3} \eto{1-4}{1-3} \mto{1-4}{1-5}
  \end{inline-diagram}  
  \end{equation}

  This picture provides a helpful guide for how to interpret the subsets $\bar X_{i+1},\bar X_i$ of $X_i$ and their complements in analogy with the algebraic setting. The inclusion $\bar X_{i+1} \mrto X_i$ is completely analogous to the inclusion of $\im(d_{i+1})$ into $X_i$, while the inclusion $X_i\backslash\bar X_i \mrto X_i$ corresponds to the inclusion of the kernel. Therefore in the picture above $\bar X_{i+1}$ represents the image of a differential $d_{i+1}$ while $\bar X_{i+1} \cup H_i(X_\dotp)$ represents the kernel of the next differential $d_i$. $\bar X_i$ in the picture can be thought of as the quotient of $X_i$ by $\ker(d_i)$.
\end{example}

As usual, a definition of homology leads to a notion of acyclic chain complexes.

\begin{definition}
    A chain complex in $\bbA$
    \begin{diagram}
    {\cdots X_{i+1} & \bar X_{i+1} & X_i & \bar X_i & X_{i-1} \cdots\\};
    \eto{1-2}{1-1} \mto{1-2}{1-3} \eto{1-4}{1-3} \mto{1-4}{1-5}
    \end{diagram} is \emph{exact} if $H_i(X_\dotp)=\varnothing$ for all $i$.
\end{definition}

Unwinding this definition, $H_i(X_\dotp)$ is $\varnothing$ exactly when each subdiagram $\bar X_{i+1} \mrto X_i \elto \bar X_i$ is a kernel-cokernel pair; that is, $\bar X_{i+1}$ is isomorphic to the kernel of $X_i \elto \bar X_i$ and $\bar X_i$ is isomorphic to the cokernel of $\bar X_{i+1} \mrto X_i$. In this case the horizontal morphism $\bar X_{i+1} \mrto X_i \bbackslash \bar X_i$ (whose cokernel is $H_i(X_\dotp)$) is an isomorphism. This condition corresponds to the chain complex being ``exact at $X_i$'' when $\bbA$ arises form an abelian category.


\begin{example}\label{exactcomplexsets}
 Expressing the above condition purely in the language of sets, we see that a chain complex of sets $X_\dotp$ is exact if and only if $X_i=\bar X_{i+1}\sqcup \bar X_i$  for all $i$. Based on this we will draw exact complexes as in \eqref{eqn.exactcomplex},
 \begin{equation}\label{eqn.exactcomplex}
 \includegraphics[height=4cm]{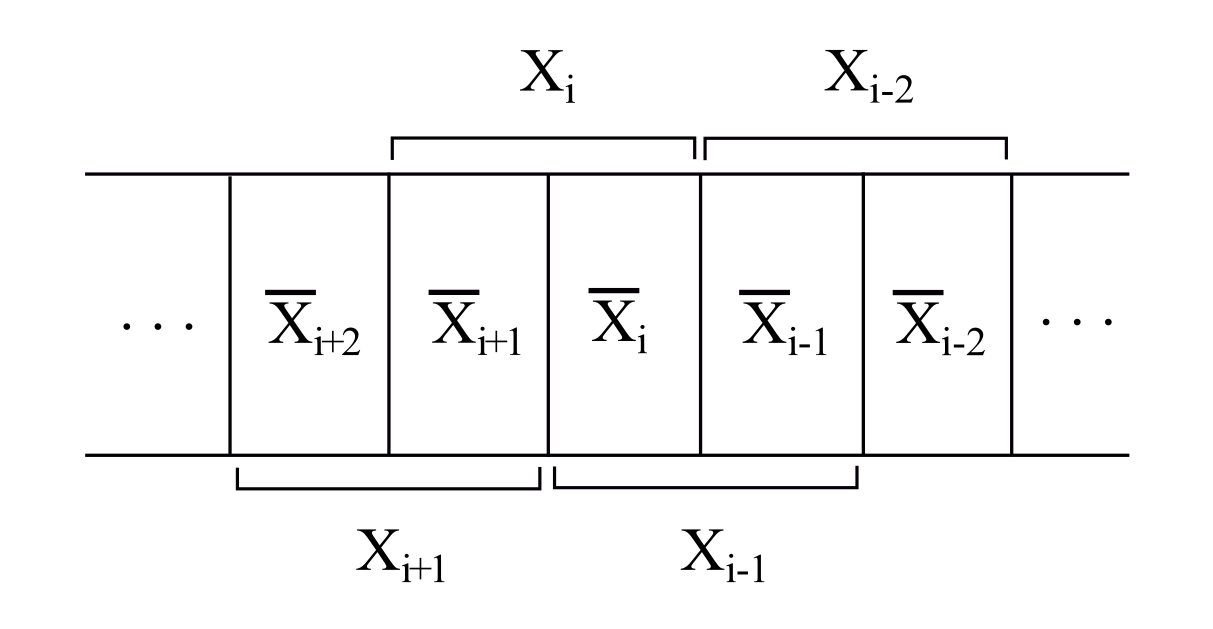}
 \end{equation}
 where the complex is determined by a sequence of ``adjacent'' but not overlapping sets $...,\bar X_{i+1},\bar X_i,...$ whose successive pairs form the sets $X_i$. 
\end{example}

\begin{example}\label{ex:shortexactseq}
    Our notion of exact chain complex is compatible with the notion of ``short exact sequence'' that we obtain from the complement functors $c$ and $k$ in our ACGW framework. Recall that in an abelian category, a short exact sequence consists of a diagram 
    \begin{diagram}
        {A & B & C\\};
        \mto{1-1}{1-2}^{i} \diagArrow{->>}{1-2}{1-3}^{p}
    \end{diagram} where $i$ is a mono, $p$ is an epi, $A=\ker p$ and $C=\coker i$. Equivalently, this is an exact chain complex concentrated on only three degrees.

    In our ACGW category of finite sets, the analogue of the first description consists of a diagram \[A\mrto^{i} B\elto^{p} C\] where both $i$ and $p$ are inclusions, $A=B\backslash C$, and $C=B\backslash A$. Then, we have that $B=A\sqcup C$, and our ``short exact sequence'' above is the same data as the exact complex concentrated on three degrees
    \begin{diagram}
    {A & A & B & C& C. \\};
    \eq{1-2}{1-1} \mto{1-2}{1-3} \eto{1-4}{1-3} \eq{1-4}{1-5}
    \end{diagram}
  \end{example}

There is a standard illustration of an exact chain complex as a ``zigzag of
short exact sequences'' which are glued together.  We can draw a similar picture
in an ACGW category as in \eqref{eqn.exactzigzag}, emphasizing how the exact complex $X_\dotp$ is built out of short exact sequences.
\begin{diagram-numbered}{eqn.exactzigzag}
  {  
    \cdots & & \bar X_{i+1} & & \bar X_i && \bar X_{i-1} && \cdots\\
    & X_{i+1} & & X_i & & X_{i-1} && X_{i-2} \\};
  \mto{1-1}{2-2} \mto{1-3}{2-4} \mto{1-5}{2-6} \mto{1-7}{2-8}
  \eto{1-3}{2-2} \eto{1-5}{2-4} \eto{1-7}{2-6} \eto{1-9}{2-8}
\end{diagram-numbered}
  
We now turn our attention to the morphisms between chain complexes.

\begin{definition}\label{defn:chainmap}
  A \emph{morphism of chain complexes} in $\bbA$, written
  $f_\dotp \colon X_\dotp \to Y_\dotp$, consists of morphisms
  $f_i: X_i \elto Z_i \mrto Y_i$ in $\bbA^\flat$ such that for each $i$
  there exists a diagram
  \begin{diagram-numbered}{}
      { X_i & \bar X_i & X_{i-1} \\
        Z_i & \bar Z_i & Z_{i-1} \\
        Y_i & \bar Y_i & Y_{i-1} \\};
      \eto{1-2}{1-1} \mto{1-2}{1-3} 
      \eto{2-2}{2-1} \mto{2-2}{2-3} 
      \eto{3-2}{3-1} \mto{3-2}{3-3} 
      
      \eto{2-1}{1-1} \mto{2-1}{3-1}
      \eto{2-2}{1-2}  \mto{2-2}{3-2} 
      \eto{2-3}{1-3} \mto{2-3}{3-3}

      \comm{2-1}{3-2} \comm{1-2}{2-3}
    \end{diagram-numbered}
    which consists of two commutative squares and two pseudo-commutative
    squares.  Note that such a diagram, if it exists, is unique: the data of the
    maps $f_i$ and the differentials in the chain complex give the outside of
    the square, and the two pseudo-commutative squares are unique (up to unique
    isomorphism).  This is exactly the condition that this square is
    commutative inside $\bbA^\flat$.

    The category of chain complexes over $\bbA$ and morphisms of chain complexes is denoted $\Ch_\bbA$.
\end{definition}

\begin{example}\label{chainmapsets}
  In the case of finite sets, this is expressing the following observation.
  Suppose that we are given four subsets $X_i$, $X_{i-1}$, $Y_i$, and $Y_{i-1}$
  of some set $A$.  Then it must be the case that
  \[(X_i\cap X_{i-1}) \cap (Y_{i} \cap Y_{i-1}) = (X_i\cap Y_i) \cap (X_{i-1}
    \cap Y_{i-1}).\]
  This quadruple intersection is exactly $\bar Z_i$.  The compatibility
  condition above is reverse-enginneering this sitation: suppose that we know
  that there exist sets $X_i$, $X_{i-1}$, $Y_i$, $Y_{i-1}$ and we're given
  pairwise-intersection data $X_i \cap X_{i-1}$, $Y_i \cap Y_{i-1}$, $X_i\cap
  Y_i$ and $X_{i-1} \cap Y_{i-1}$ (with the additional assumption that $X_i \cap
  Y_{i-1}$ and $X_{i-1}\cap Y_i$ are as small as possible).  Is it possible that
  such an $A$ exists?  The answer is yes exactly when a diagram as above exists.

  Pictorially, we can view a morphism of chain complexes as on the left of \eqref{eqn.setchainmorphism}, where each oval consists of the sets $X_i$ and $Y_i$ overlapping on $Z_i$ as illustrated on the right of \eqref{eqn.setchainmorphism}.
  \begin{equation}\label{eqn.setchainmorphism}
      \raisebox{-.45\height}{\includegraphics[height=4.5cm]{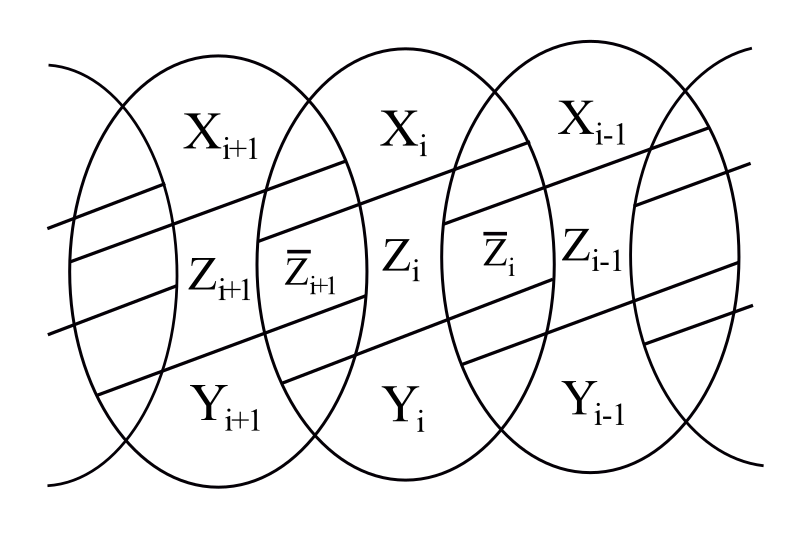}}\qquad\qquad\raisebox{-.45\height}{\includegraphics[height=4.5cm]{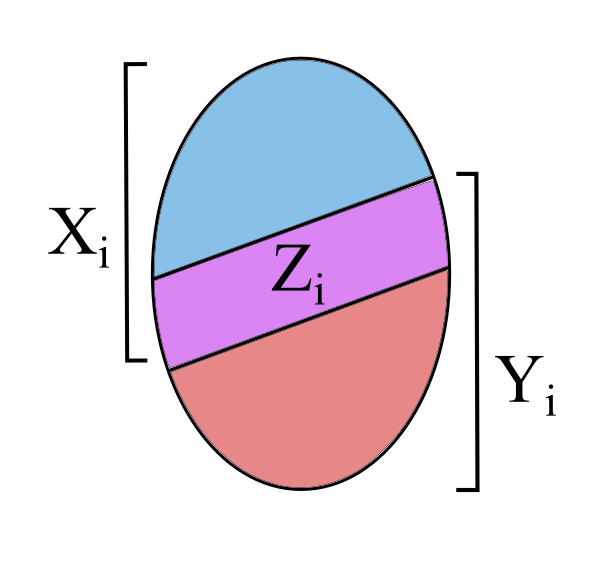}}
  \end{equation}
  While this picture is somewhat complicated, we will see that it becomes easier to work with the pictures (and the morphisms themselves) when horizontal and vertical morphisms of chain complexes are treated separately rather than combined into a span.

  The category $\Ch_{\FFinSet}$ is exactly $i\Ch_\FinSet$. 
\end{example}

Once again, we compare this definition to the usual notion for abelian categories, and then discuss how it can be interpreted for sets using partial functions.

\begin{example}
    In an abelian category $\mathcal{A}$, a morphism of chain complexes
    $f_\dotp\colon X_\dotp\to Y_\dotp$ consists of a sequence of morphism
    $f_i:X_i \rto Y_i$ such that for all $i$ the diagram
    \begin{diagram}
        {X_i & X_{i-1}\\
        Y_i  & Y_{i-1}\\};
        \to{1-1}{1-2}^{d_i} \to{1-1}{2-1}_{f_i} \to{1-2}{2-2}^{f_{i-1}} \to{2-1}{2-2}_{d'_i}
      \end{diagram} commutes.  To compare it with our ACGW definition, we
      factor the maps as epi-mono in order to express things in ACGW
      language. First, we can factor the differentials to obtain a commutative
      diagram as in \eqref{eqn.absquarefactor}.
    \begin{equation}\label{eqn.absquarefactor}
        \begin{tikzcd}
        X_i\ar[rr,"d_i"]\ar[rd,"p_i",->>]\ar[ddd,"f_i"] & & X_{i-1}\ar[ddd,"f_{i-1}"]\\
        &\bar X_i\ar[ur,"j_i",>->]\ar[d, dashed, "\bar f_{i-1}"] &\\
        & \bar Y_i\ar[rd,"j'_i",>->] &\\
        Y_i\ar[ur,"p'_i",->>]\ar[rr,"d'_i"] & & Y_{i-1}
    \end{tikzcd}
    \end{equation} 
    The dashed map $\bar f_{i-1}$ is simply the restriction of $f_{i-1}$ to $\bar X_i=\im d_i$; the fact that $f_{i-1}d_i=d'_i f_i$ ensures that the image of $\bar f_{i-1}$ lies inside of $\bar Y_i=\im d'_i$ as implicitly claimed in the diagram. Then, we further factor the vertical morphisms as depicted in the commutative diagram in \eqref{eqn.absquaremorefactor}.
    \begin{diagram-numbered}{eqn.absquaremorefactor}
    { X_i & \bar X_i & X_{i-1} \\
     Z_i & \bar Z_i & Z_{i-1} \\
     Y_i & \bar Y_i & Y_{i-1} \\};
    \diagArrow{->>}{1-1}{1-2}^{p_i} \mto{1-2}{1-3}^{j_i}
    \diagArrow{dashed,->}{2-1}{2-2}^{\bar p'_i} \diagArrow{dashed,->}{2-2}{2-3}^{\bar j'_i} 
    \diagArrow{->>}{3-1}{3-2}_{p'_i} \mto{3-2}{3-3}_{j'_i} 

    \diagArrow{->>}{1-1}{2-1}^{f^1_{i}} \mto{2-1}{3-1}_{f^2_{i}}
    \diagArrow{->>}{1-2}{2-2}^{\bar f^1_{i-1}} \mto{2-2}{3-2}^{\bar f^2_{i-1}}
    \diagArrow{->>}{1-3}{2-3}^{f^1_{i-i}} \mto{2-3}{3-3}^{f^2_{i-i}}
    \end{diagram-numbered}
     The dashed maps are the restrictions of the corresponding maps to the images, just as in the previous diagram. Finally, we have that $\bar p'_i$ is an epi, since $\bar p'_i f^1_i=\bar f^1_{i-1}p_i$ which is an epi, and similarly $\bar j'_i$ is a mono, as $f^2_{i-1}\bar j'_i=j'_i \bar f^2_{i-1}$ which is a mono. We can now see that this picture matches precisely our definition of morphism of complexes of sets.
\end{example}

\begin{remark}
    We can also interpret our morphisms of chain complexes of sets by interpreting the morphisms in $\FinSet^\flat$ as injective partial functions, where the backwards monomorphism is the inclusion of the domain of a partial function. Suppose that we are
    given a square of injective partial functions as below.
    \begin{diagram}
        {X_i & X_{i-1}\\
        Y_i  & Y_{i-1}\\};
        \to{1-1}{1-2}^{d_i} \to{1-1}{2-1}_{f_i} \to{1-2}{2-2}^{f_{i-1}} \to{2-1}{2-2}_{d'_i}
    \end{diagram}
    The composite $f_{i-1}d_i$ is given by the map 
    \begin{equation*}
        x\mapsto \begin{cases} * \text{ if } x\in X_i\backslash \bar X_i \\ d_i x \text{ if } x\in\bar X_i\end{cases} \mapsto \begin{cases} * \text{ if } x\in X_i\backslash\bar X_i \\ * \text{ if } x\in \bar X_i \text{ and } d_i x\in X_{i-1}\backslash Z_{i-1} \\ f_{i-1}d_{i-1}x \text{ if } x\in\bar X_i \text{ and } d_ix\in Z_{i-1}\end{cases}
    \end{equation*} where $\bar X_i$ denotes the partial domain of $d_i$ and $Z_{i-1}$ the partial domain of $f_{i-1}$. Then, the partial domain of the composite $f_{i-1}d_i$ is precisely the pullback 
    \begin{diagram}
    { \bar Z_i &  Z_{i-1} \\
      \bar X_{i} & X_{i-1} \\};
      \mto{2-1}{2-2} \eto{1-2}{2-2} \mto{1-1}{1-2} \eto{1-1}{2-1}
      \pbsym{1-1}{2-2};
\end{diagram} The requirement that the chain map $f$ commutes with the differentials amounts to saying that the composites $f_{i-1}d_i$ and $d'_i f_i$ have the same partial domain and agree on that domain, which can be expressed as the diagram given in Definition \ref{defn:chainmap}.
\end{remark}

Now that we have a definition of a morphism of chain complexes, we can use it to
define an ACGW structure on $\Ch_\bbA$.

\begin{definition}
  The ACGW category $\CCh_\bbA$ has
  \begin{description}
  \item[objects] objects of $\Ch_\bbA$,
  \item[horizontal morphisms] a horizontal morphism $X_\dotp \mrto Y_\dotp$ is a
    sequence of horizontal morphisms $X_i \mrto Y_i$ such that there exists a
    horizontal morphism $\bar X_i \mrto \bar Y_i$ such that the resulting
    diagram
    \begin{diagram-numbered}{}
      { X_i & \bar X_i & X_{i-1} \\
        Y_i & \bar Y_i & Y_{i-1} \\};
      \eto{1-2}{1-1} \mto{1-2}{1-3}
      \eto{2-2}{2-1} \mto{2-2}{2-3}
      \mto{1-1}{2-1} \mto{1-2}{2-2} \mto{1-3}{2-3}
      \comm{1-1}{2-2}
    \end{diagram-numbered}
    consists of a pseudo-commutative square and a commutative square.
  \item[vertical morphisms] a vertical morphism $Z_\dotp \erto Y_\dotp$ is a
    sequence of vertical morphisms $Z_i \erto Y_i$ such that there exists a
    vertical morphism $\bar Z_i \erto \bar Y_i$ such that the resulting diagram
    \begin{diagram-numbered}{}
      { Z_i & \bar Z_i & Z_{i-1} \\
        Y_i & \bar Y_i & Y_{i-1} \\};
      \eto{1-2}{1-1} \mto{1-2}{1-3}
      \eto{2-2}{2-1} \mto{2-2}{2-3}
      \eto{1-1}{2-1} \eto{1-2}{2-2} \eto{1-3}{2-3}
      \comm{1-2}{2-3}
    \end{diagram-numbered}
    consists of a commutative square and a pseudo-commutative square.
  \item[pseudo-commutative squares] those that are pseudo-commutative at each
    $i$ and each ``halfway point.''
  \end{description}
\end{definition}

We choose not include the proof of the ACGW axioms, which is very technical and proceeds analogously to that of \cite[Theorem 5.6.16]{thesis}.

\begin{example}
    In the case of an abelian category, the horizontal and vertical morphisms of chain complexes are natural monomorphisms and (reversed) epimorphisms. In fact more generally, the horizontal and vertical morphisms of chain complexes are exactly the morphisms of chain complexes which, rather than consisting of a span at each level, contain only horizontal or vertical morphisms.
\end{example}

\begin{example}\label{setchainmaps}
    In the case of sets, horizontal and vertical morphisms of chain complexes can be pictured as in \eqref{eqn.setmemorphisms} on the left and right respectively. 
    \begin{equation}\label{eqn.setmemorphisms}
    \begin{tikzpicture}[baseline]%
      \matrix (m) [matrix of math nodes,row sep=2.6em, column sep=2.8em]
      { X_i & \bar X_i & X_{i-1} \\
        Y_i & \bar Y_i & Y_{i-1} \\};
      \eto{1-2}{1-1} \mto{1-2}{1-3}
      \eto{2-2}{2-1} \mto{2-2}{2-3}
      \mto{1-1}{2-1} \mto{1-2}{2-2} \mto{1-3}{2-3}
      \comm{1-1}{2-2}
    \end{tikzpicture}\qquad\qquad\qquad\begin{tikzpicture}[baseline]%
      \matrix (m) [matrix of math nodes,row sep=2.6em, column sep=2.8em]
      { Z_i & \bar Z_i & Z_{i-1} \\
        Y_i & \bar Y_i & Y_{i-1} \\};
      \eto{1-2}{1-1} \mto{1-2}{1-3}
      \eto{2-2}{2-1} \mto{2-2}{2-3}
      \eto{1-1}{2-1} \eto{1-2}{2-2} \eto{1-3}{2-3}
      \comm{1-2}{2-3}
    \end{tikzpicture}
    \end{equation}
    \[
    \raisebox{-.45\height}{\includegraphics[height=6cm]{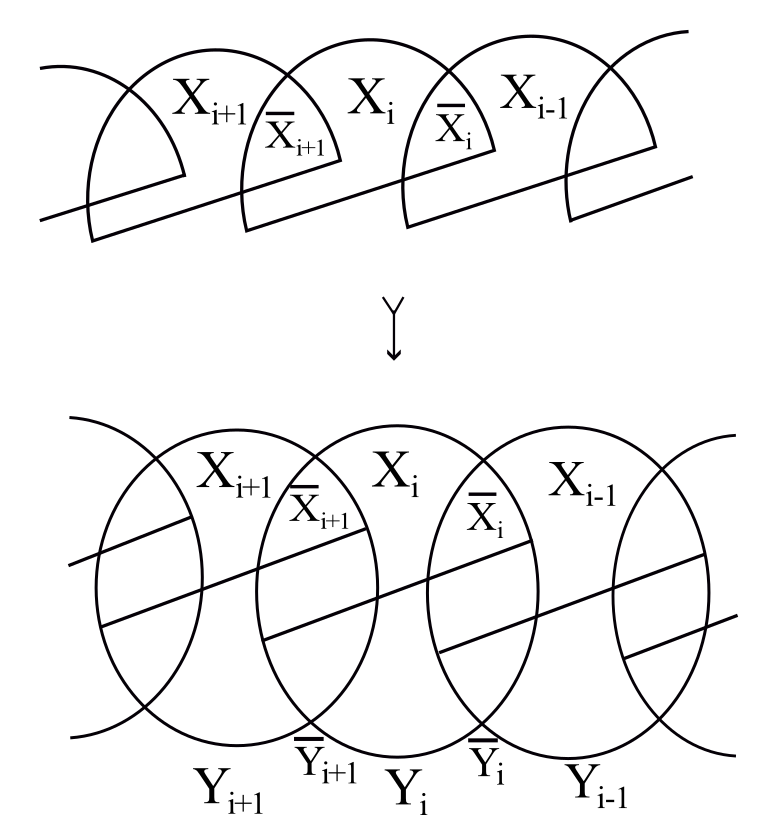}}\qquad\qquad\raisebox{-.45\height}{\includegraphics[height=6cm]{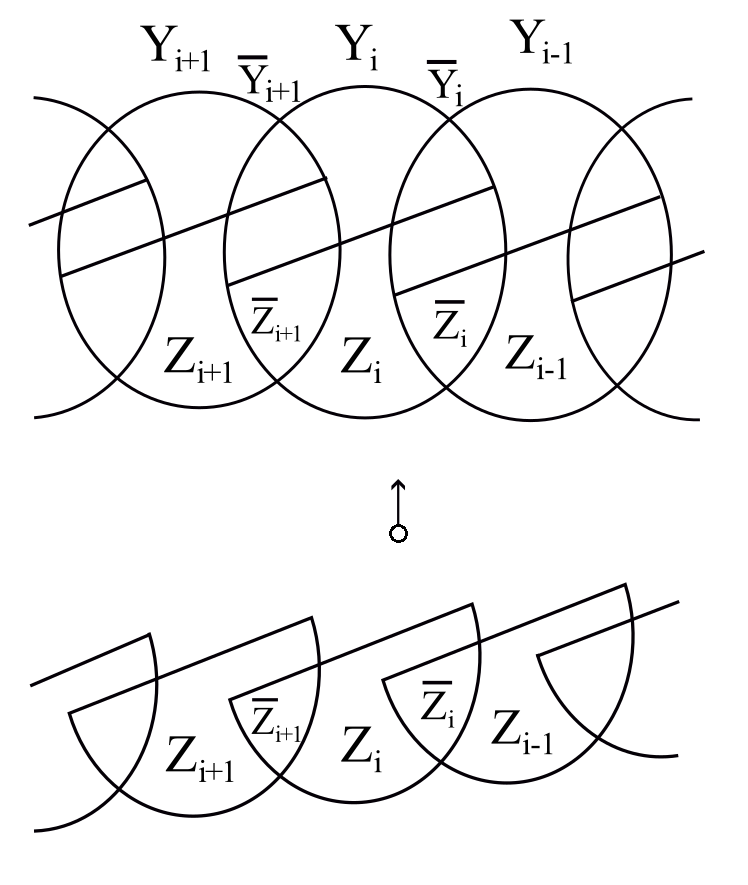}}
    \]
    In both cases there is a levelwise inclusion of chain complexes which commutes with the inclusions in the differentials, but horizontal morphisms require $\bar X_i$ to be the entire intersection of $X_i$ and $\bar Y_i$, while vertical morphisms instead require $\bar Z_i$ to be the intersection of $Z_{i-1}$ and $\bar Y_i$. These pictures provide a convenient tool for keeping track of this, where in a chain complex pictured as overlapping circles the image of a horizontal morphism behaves as the top half of the up-right directed lines and the image of a vertical morphism behaves as the bottom half.

    These pictures also illustrate how to take the complement of a horizontal or vertical morphism, as the inclusion of the levelwise complement at each level is evidently of the opposite form to the original morphism. Moreover, a chain complex $Y_\dotp$ partitioned as in \eqref{eqn.setSES} can be regarded as a short exact sequence $X_\dotp \mrto Y_\dotp \elto Z_\dotp$ (or "extension" or "complementary pair").
    \begin{equation}\label{eqn.setSES}
    \raisebox{-.45\height}{\includegraphics[height=8cm]{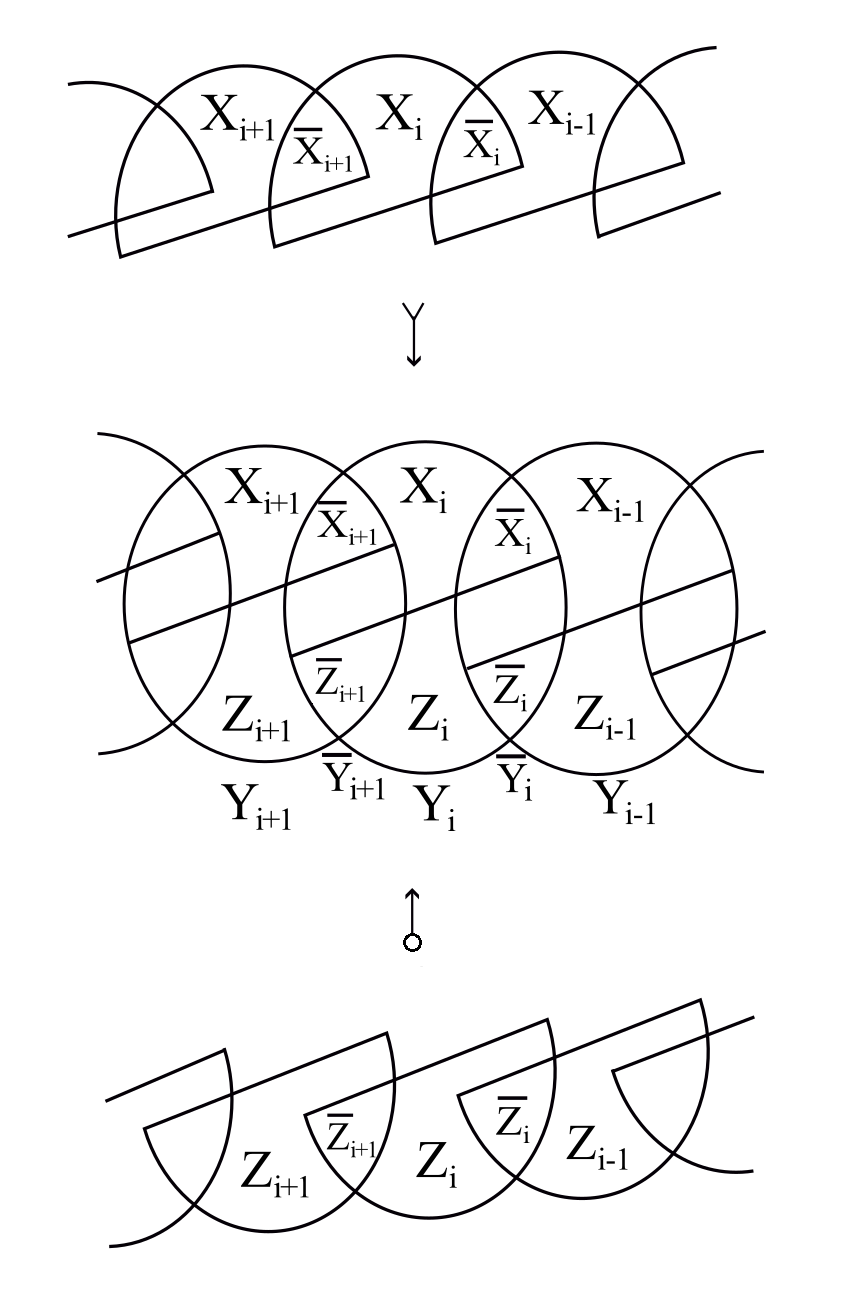}}
    \end{equation}
    These complements are not computed levelwise on the intersections, as each $\bar Y_i$ may include more than the union of $\bar X_i$ and $\bar Z_i$, but as the picture shows in the complement of the horizontal morphism $X_\dotp \mrto Y_\dotp$ the intersections can be computed by intersecting $\bar Y_i \backslash \bar X_i$ with $Y_{i-1} \backslash X_{i-1}$, and similarly for the complement of a vertical morphism. 
\end{example}

We conclude this section with an alternative description of a morphism of chain complexes that factors them into horizontal and vertical morphisms. As we've explained previously, our spans of sets $X \elto Z \mrto Y$ encode partially defined inclusions $X\to Y$. The same is true for chain complexes.

\begin{lemma}\label{lemma:chainmap}
    A morphism of chain complexes $f_\dotp \colon X_\dotp\to Y_\dotp$ consists
    of a span of inclusions of chain complexes \[X_\dotp \elto Z_\dotp \mrto
      Y_\dotp.\]  More formally,
    \[(\CCh_\bbA)^\flat \cong \Ch_\bbA.\]
  \end{lemma}
\begin{proof}
The proof of this is straightforward, amounting only to checking that the middle row of the diagram in \eqref{eqn.proofchainmorphism} defining a morphism of chain complexes $f_\dotp$ does in fact satisfy the chain condition.
   \begin{diagram-numbered}{eqn.proofchainmorphism}
    {X_i & \bar X_i & X_{i-1} \\
     Z_i & \bar Z_i & Z_{i-1} \\
     Y_i & \bar Y_i & Y_{i-1} \\};
    \eto{1-2}{1-1} \mto{1-2}{1-3} 
    \eto{2-2}{2-1} \mto{2-2}{2-3} 
    \eto{3-2}{3-1} \mto{3-2}{3-3} 

    \eto{2-1}{1-1}^{f^1_{i}} \mto{2-1}{3-1}_{f^2_{i}}
    \eto{2-2}{1-2}^{\bar f^1_{i-1}} \mto{2-2}{3-2}_{\bar f^2_{i-1}}
    \eto{2-3}{1-3}^{f^1_{i-i}} \mto{2-3}{3-3}_{f^2_{i-i}}

    \path[font=\scriptsize] (m-2-2) edge[-,white] node[pos=0.05]
  {\textcolor{black}{$\urcorner$}}  (m-1-3);
  \path[font=\scriptsize] (m-2-2) edge[-,white] node[pos=0.05]
  {\textcolor{black}{$\llcorner$}}  (m-3-1);
    \end{diagram-numbered}
\end{proof}

\section{Recovering classical results}\label{section:results}

In this section we study the notions of chain complexes, exact sequences, and homology of finite sets introduced in Section \ref{section:homologydefn} to recover several classical results present in the context of abelian categories. 

\subsection{The Snake Lemma}

Our goal in this subsection is to prove a version of the Snake Lemma for finite sets. We begin by recalling the classical statement of the Snake Lemma.

\begin{theorem}[Snake Lemma] Let $\mathcal{A}$ be an abelian category and consider a morphism of exact sequences as in \eqref{eqn.oldsnake}
\begin{diagram-numbered}{eqn.oldsnake}
    { & A & B & C & 0\\
    0 & A' & B' & C'\\};
    \diagArrow{->}{1-2}{1-3} \diagArrow{->>}{1-3}{1-4} \diagArrow{->}{1-4}{1-5}
    \diagArrow{->}{2-1}{2-2} \mto{2-2}{2-3} \diagArrow{->}{2-3}{2-4}
    \diagArrow{->}{1-2}{2-2}^{f} \diagArrow{->}{1-3}{2-3}^{g} \diagArrow{->}{1-4}{2-4}^{h}
\end{diagram-numbered}
 Then, there exists an exact sequence \[\ker f\to \ker g\to \ker h \to \coker f\to\coker g \to\coker h.\] Moreover, if $A\to B$ is a mono, then so is $\ker f\to\ker g$, and if $B'\to C'$ is an epi, then so is $\coker g\to\coker h$.
\end{theorem}

In the case of finite sets, it will be easier to prove the version with a stronger assumption first, where we start with two short exact sequences. Using the description of short exact sequences from  Example \ref{ex:shortexactseq} and our definition of morphisms, we find that the statement of the Snake Lemma translates to the following. 

\begin{theorem}[Weak Snake Lemma for ACGW categories]\label{thm:weaksnakelemma}
Given a morphism of short exact sequences in $\bbA$ as in \eqref{eqn.weaksnakeinput},
    \begin{diagram-numbered}{eqn.weaksnakeinput}
    {A  & B & C\\
    X & Y & Z\\
     A' & B' & C'\\};
    \mto{1-1}{1-2} \eto{1-3}{1-2} 
    \mto{2-1}{2-2} \eto{2-3}{2-2} 
    \mto{3-1}{3-2} \eto{3-3}{3-2} 

    \eto{2-1}{1-1} \mto{2-1}{3-1}
    \eto{2-2}{1-2} \mto{2-2}{3-2}
    \eto{2-3}{1-3} \mto{2-3}{3-3}

    \comm{1-1}{2-2} \comm{2-2}{3-3}
    \end{diagram-numbered}
    there is an exact sequence of the form in \eqref{eqn.weaksnakesequence}.
    \begin{diagram-numbered}{eqn.weaksnakesequence}
    {  &  &[-10] D &[-10] &[-10] W &[-10] &[-10] D' &[-10] \\
    A\bbackslash X  & B\bbackslash Y & & C\bbackslash Z & & A' \fforwardslash  X & & B' \fforwardslash  Y & C' \fforwardslash  Z\\};
    \mto{2-1}{2-2} \eto{1-3}{2-2} 
    \mto{1-3}{2-4} \eto{1-5}{2-4} \mto{1-5}{2-6}
    \eto{1-7}{2-6} \mto{1-7}{2-8} \eto{2-9}{2-8}
    \end{diagram-numbered}
\end{theorem}

Before we prove this, we can summarize the essence of the proof with the picture on the left in \eqref{eqn.setsnake} which shows this morphism of short exact sequences in the ACGW category $\FFinSet$.
\begin{equation}\label{eqn.setsnake}
\raisebox{-.5\height}{\includegraphics[height=4cm]{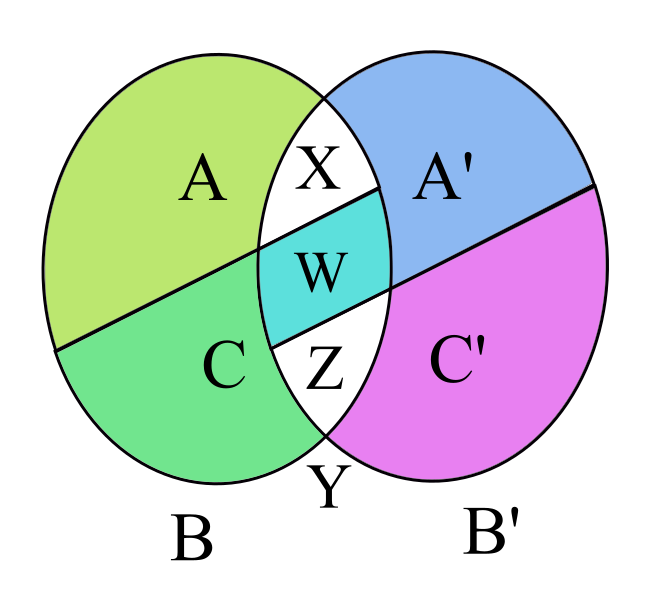}}
\qquad\qquad
\raisebox{-.45\height}{\includegraphics[height=3cm]{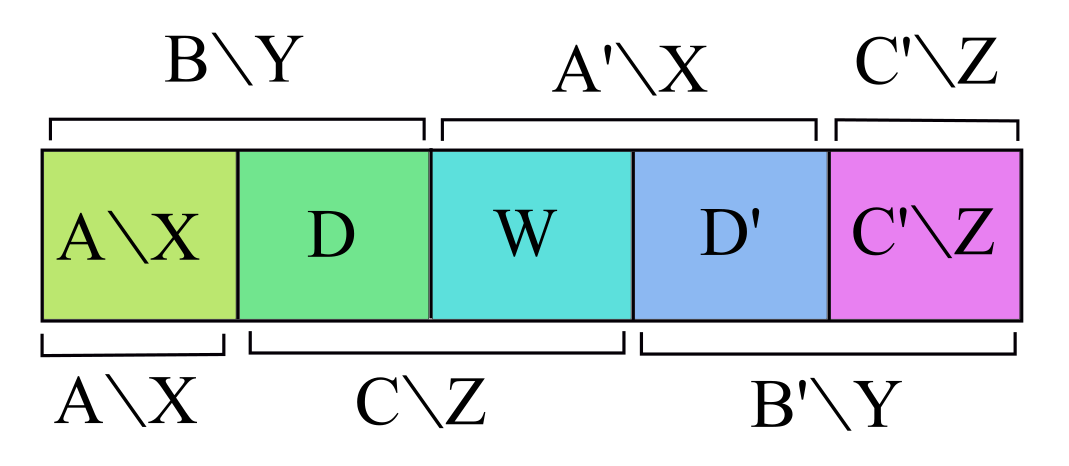}}
\end{equation}
The exact sequence, which as in Example \ref{exactcomplexsets} is a sequence of subsets in the diagram whose successive unions agree with the bottom row of the zigzag diagram in the theorem, is displayed in the pictures as the sequence $A \backslash X, D=C \backslash (Y \backslash X), W=(Y \backslash X) \backslash Z, D'=A' \backslash (Y \backslash Z), C' \backslash Z$. 

The goal of the proof is therefore to construct $D,W,D'$ in the ACGW language along with the desired maps to $B \bbackslash Y, C \bbackslash Z, A'  \fforwardslash  X,B'  \fforwardslash  Y$.

\begin{proof}
    Taking the kernel and cokernel of the top-left square in the diagram in both directions we get a diagram of pullbacks and pseudo-commutative squares as in \eqref{eqn.prooffactoredmorphism},
    \begin{diagram-numbered}{eqn.prooffactoredmorphism}
        { A \bbackslash X & B \bbackslash Y & D & C \bbackslash Z\\
          A & B & C & \\
          X & Y & Y  \fforwardslash  X & Z \\ };
        \comm{1-2}{2-3} \comm{2-1}{3-2}
        \mto{1-1}{1-2} \mto{2-1}{2-2} \mto{3-1}{3-2}
        \mto{1-1}{2-1} \mto{1-2}{2-2} \mto{1-3}{2-3}
        \eto{1-3}{1-2} \eto{2-3}{2-2} \eto{3-3}{3-2}
        \eto{3-1}{2-1} \eto{3-2}{2-2} \eto{3-3}{2-3}
        \eto{3-4}{3-3} \eto{3-4}{2-3} 
        \mto{1-3}{1-4} \mto{1-4}{2-3}
    \end{diagram-numbered}
    where the top row contains the beginning of our desired exact sequence. The horizontal morphism $D \mrto C \bbackslash Z$ is induced by the vertical morphism $Y  \fforwardslash  X \elto Z$ as $D \cong C \bbackslash (Y  \fforwardslash  X)$, based on the diagram  of kernels and cokernels in \eqref{eqn.210}, where we denote $W=(Y\fforwardslash X)\bbackslash Z$.\footnote{See \cite[Lemma 2.10]{CZ-cgw} for more detail.}
    \begin{diagram-numbered}{eqn.210}
        { D & C \bbackslash Z & W \\
          D & C & Y  \fforwardslash  X \\
            & Z & Z \\ };
        \eq{1-1}{2-1} \eq{3-2}{3-3}
        \mto{1-1}{1-2} \mto{2-1}{2-2}
        \mto{1-2}{2-2} \mto{1-3}{2-3}
        \eto{1-3}{1-2} \eto{2-3}{2-2}
        \eto{3-2}{2-2} \eto{3-3}{2-3}
        \comm{1-2}{2-3}
      \end{diagram-numbered}
      In order to show that the kernel of $W \erto C\bbackslash
      Z$ is naturally isomorphic to the kernel of $Y \fforwardslash X \erto C$, it suffices to check that
      the cokernel of $W \mrto Y \fforwardslash X$ is naturally isomorphic to the cokernel of $C
      \bbackslash Z \mrto C$.  (This is because the upper-right square is
      ``distinguished''; see \cite[Section 2]{CZ-cgw} for a more in-depth discussion.)
      As this is true, we are done; note that the sequence so far is exact by construction.

    The construction of $D'$ is entirely dual to this, namely as $A'  \fforwardslash  (Y \bbackslash Z)$. All that remains then is to show that $W$ can be equivalently constructed as $(Y  \fforwardslash  X) \bbackslash Z$ as above or dually as $(Y \bbackslash Z)  \fforwardslash  X$, but these are the two equivalent definitions of the homology of the sequence $X \mrto Y \elto Z$ from Definition \ref{defn:homology}, so the proof is complete.
\end{proof}

The stronger form of the snake lemma in an ACGW category is as follows, where the short exact sequences are replaced by sequences which are no longer exact at $A$ and $C'$.

\begin{theorem}[Strong Snake Lemma for ACGW categories]\label{thm:snakelemma}
Given a morphism of exact sequences in an ACGW category $\bbA$ as in \eqref{eqn.strongsnakeinput}, 
    \begin{diagram-numbered}{eqn.strongsnakeinput}
    {A & \bar A  & B & C & C\\
    X & X & Y & Z & Z\\
     A' & A' & B' & \bar C' & C'\\};
    \eto{1-2}{1-1} \mto{1-2}{1-3} \eto{1-4}{1-3} \eq{1-4}{1-5}
    \eq{2-2}{2-1} \mto{2-2}{2-3} \eto{2-4}{2-3} \eq{2-4}{2-5}
    \eq{3-1}{3-2} \mto{3-2}{3-3}  \eto{3-4}{3-3} \mto{3-4}{3-5}

    \eto{2-1}{1-1} \mto{2-1}{3-1}
    \eto{2-2}{1-2} \mto{2-2}{3-2}
    \eto{2-3}{1-3} \mto{2-3}{3-3}
    \eto{2-4}{1-4} \mto{2-4}{3-4}
    \eto{2-5}{1-5} \mto{2-5}{3-5}

    \comm{1-2}{2-3} \comm{2-3}{3-4} \comm{2-1}{3-2} \comm{1-4}{2-5}
    \end{diagram-numbered}
    there is an exact sequence of the form in \eqref{eqn.strongsnakesequence}.
    \begin{diagram-numbered}{eqn.strongsnakesequence}
    { &[-25] \bar A \bbackslash X &[-25]  &[-15] D &[-15] &[-15] W &[-15] &[-15] D' &[-15] &[-25] \bar C'  \fforwardslash  Z &[-25] \\
    A\bbackslash X  & & B\bbackslash Y & & C\bbackslash Z & & A'  \fforwardslash  X & & B'  \fforwardslash  Y & & C'  \fforwardslash  Z\\};
     \eto{1-2}{2-1}\mto{1-2}{2-3} \eto{1-4}{2-3} 
    \mto{1-4}{2-5} \eto{1-6}{2-5} \mto{1-6}{2-7}
    \eto{1-8}{2-7} \mto{1-8}{2-9} \eto{1-10}{2-9} \mto{1-10}{2-11}
    \end{diagram-numbered}
\end{theorem}

In this case, the picture in $\FFinSet$ is very similar to that for the weak version of the snake lemma but allows for the exact sequence to continue beyond the union of $B$ and $B'$.
\begin{equation}\label{eqn.strongsnake}
\raisebox{-.5\height}{\includegraphics[height=4cm]{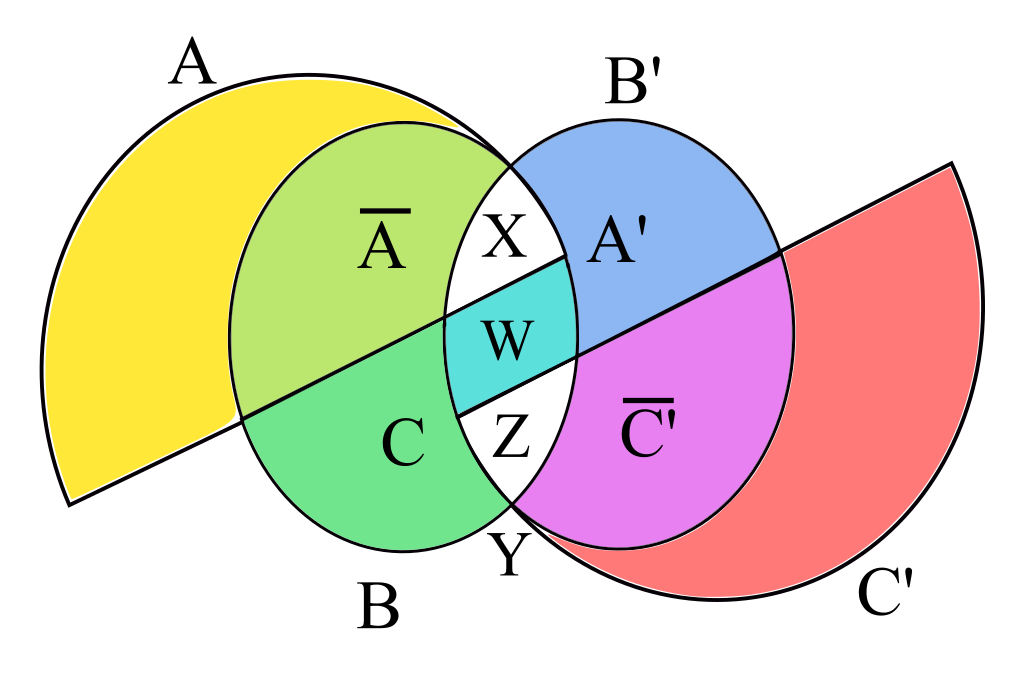}}
\qquad
\raisebox{-.45\height}{\includegraphics[height=2.7cm]{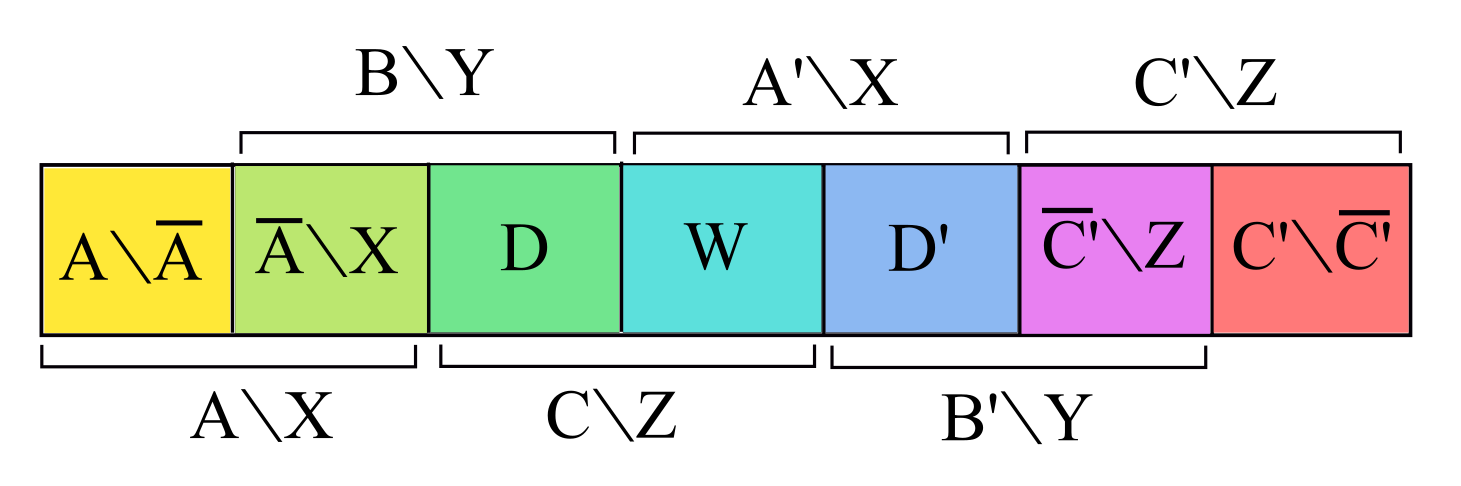}}
\end{equation}

\begin{proof}
    Applying Theorem \ref{thm:weaksnakelemma} to the middle two columns of the diagram gives all but the leftmost and rightmost morphisms in the desired exact sequence. The remaining maps are those induced on kernels and cokernels by the pullback squares in the upper left and lower right corners of the diagram in the theorem statement.
\end{proof}

\subsection{Long exact sequence in homology}

In this subsection, we will show how a short exact sequence of chain complexes induces a long exact sequence on homology. In order to be more precise, and to get the correct formulation in the context of finite sets, we start by recalling the classical result for abelian categories.

\begin{theorem}[Long exact sequence in homology]
Let $\mathcal{A}$ be an abelian category. Then, every short exact sequence of chain complexes in $\mathcal{A}$ \[0\to X_\dotp\rcofib Y_\dotp \rfib Z_\dotp\to 0\] induces a long exact sequence in homology \[\cdots \to H_{i+1}(Z_\dotp)\to H_i(X_\dotp)\to H_i(Y_\dotp)\to H_i(Z_\dotp)\to H_{i-1}(X_\dotp)\to\cdots.\]
\end{theorem}

Recall that a short exact sequence of chain complexes in an ACGW category $\bbA$ is a pair of morphisms $X_\dotp \mrto Y_\dotp \elto Z_\dotp$ such that $X_\dotp$ is isomorphic to the kernel of the vertical morphism or, equivalently, $Z_\dotp$ is isomorphic to the cokernel of the horizontal morphism. As we discussed in Example \ref{setchainmaps}, in $\FFinSet$ this short exact sequence can be visualized as in the picture in \eqref{eqn.setSES2}, where each set $Y_i$ is partitioned into $X_i$ and $Z_i$ in a manner consistent with the respective horizontal and vertical morphisms of chain complexes.
\begin{equation}\label{eqn.setSES2}
\raisebox{-.5\height}{\includegraphics[height=4cm]{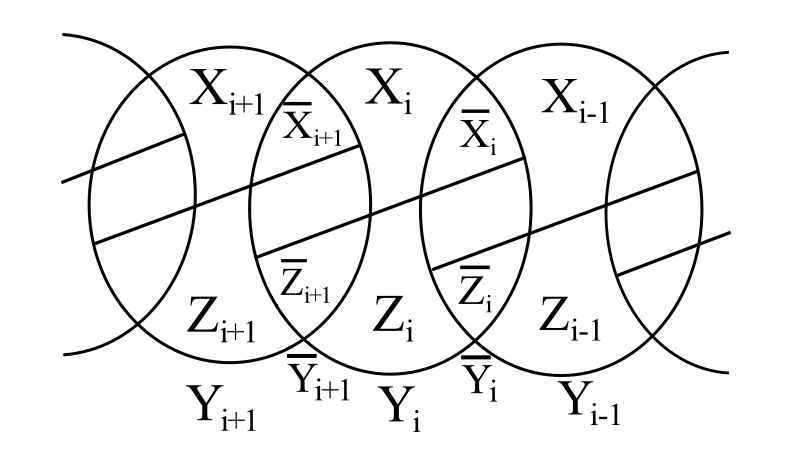}}
\end{equation}

 \begin{theorem}[Long exact sequence in homology for ACGW categories]\label{thm:lesforhomology}
 An exact sequence of chain complexes in $\bbA$ \[ X_\dotp \mrto Y_\dotp \elto Y_\dotp\fforwardslash X_\dotp \]
     induces a long exact sequence (i.e. an exact chain complex) in homology as in \eqref{eqn.LESpair}.
    \begin{diagram-numbered}{eqn.LESpair}
    { &[-30] \dotp &[-10]  &[-10] \dotp &[-10] &[-10] \dotp &[-20] &[-20] \dotp &[-20]  \\
    \cdots H_{i+1}(Y_\dotp \fforwardslash X_\dotp)  & & H_i(X_\dotp) & & H_i(Y_\dotp) & & H_i(Y_\dotp \fforwardslash X_\dotp) & & H_{i-1}(X_\dotp)\cdots\\};
     \eto{1-2}{2-1}\mto{1-2}{2-3} \eto{1-4}{2-3} 
    \mto{1-4}{2-5} \eto{1-6}{2-5} \mto{1-6}{2-7}
    \eto{1-8}{2-7} \mto{1-8}{2-9}
    \end{diagram-numbered}
 \end{theorem}

 Setting $Z_\dotp \coloneqq Y_\dotp \fforwardslash X_\dotp$, this exact sequence can be visualized as in the pictures in \eqref{eqn.setLESpair}.
 \begin{equation}\label{eqn.setLESpair}
 \raisebox{-.5\height}{\includegraphics[height=4cm]{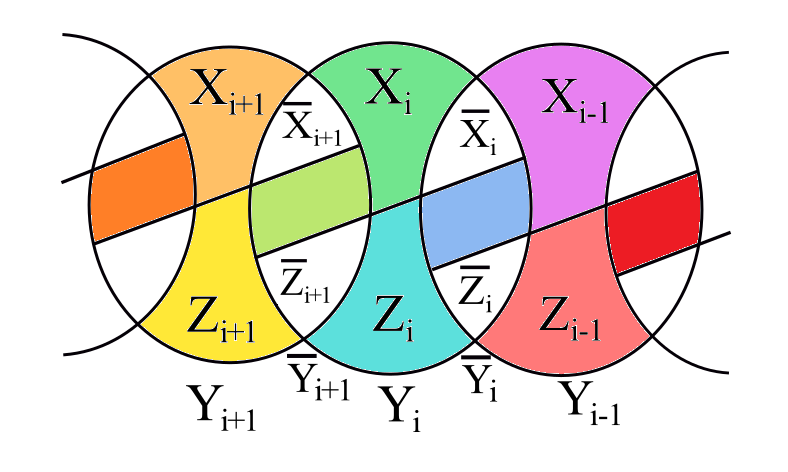}}
 \qquad
 \raisebox{-.45\height}{\includegraphics[height=3cm]{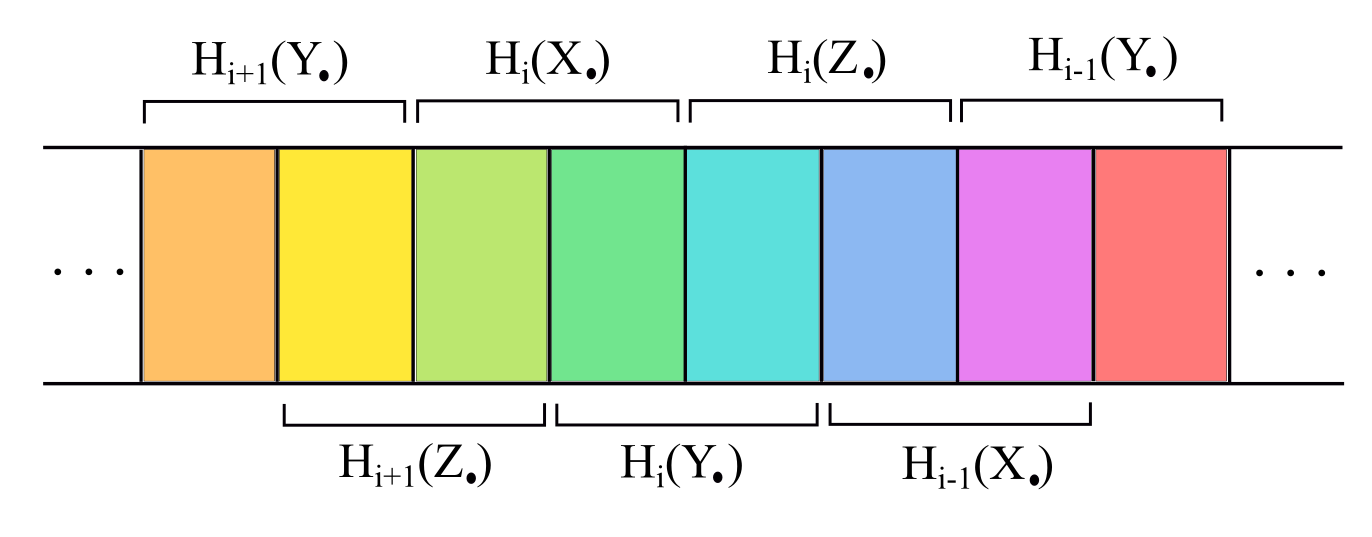}}
 \end{equation}
 
 \begin{proof}
    Let $Z_\dotp \coloneqq Y_\dotp \fforwardslash X_\dotp$. The data making up the exact sequence includes the diagram in \eqref{eqn.chainLESdiagram},
    \begin{diagram-numbered}{eqn.chainLESdiagram}
        { \bar X_{i+1} & \bar Y_{i+1} \bbackslash \bar Z_{i+1} & \bar Y_{i+1} & \bar Z_{i+1} & \\
          & X_i & Y_i & Z_i & \\
          & \bar X_i & \bar Y_i & \bar Z_i & \\
          & X_{i-1} & Y_{i-1} & Z_{i-1} & \\
          & \bar X_{i-1} & \bar Y_{i-1} & \bar Y_{i-1}  \fforwardslash  \bar X_{i-1} & \bar Z_{i-1} \\};
        \pbsym{1-2}{2-3}\comm{1-3}{2-4} \comm{2-2}{3-3}
        \comm{3-3}{4-4} \comm{4-2}{5-3} \path[font=\scriptsize] (m-4-3) edge[-,white] node[pos=0.05]   {\textcolor{black}{$\ulcorner$}}  (m-5-4);
        \mto{1-1}{1-2} \mto{1-1}{2-2} 
        \eto{5-5}{5-4} \eto{5-5}{4-4}
        \mto{1-2}{1-3} \eto{1-4}{1-3}
        \mto{2-2}{2-3} \eto{2-4}{2-3}
        \mto{3-2}{3-3} \eto{3-4}{3-3}
        \mto{4-2}{4-3} \eto{4-4}{4-3}
        \mto{5-2}{5-3} \eto{5-4}{5-3}
        \mto{1-2}{2-2} \mto{1-3}{2-3} \mto{1-4}{2-4}
        \eto{3-2}{2-2} \eto{3-3}{2-3} \eto{3-4}{2-4}
        \mto{3-2}{4-2} \mto{3-3}{4-3} \mto{3-4}{4-4}
        \eto{5-2}{4-2} \eto{5-3}{4-3} \eto{5-4}{4-4}
    \end{diagram-numbered}
    which by taking cokernels of the top row and kernels of the bottom row induces a diagram as in \eqref{eqn.LESsnakediagram}.
     \begin{diagram-numbered}[2em]{eqn.LESsnakediagram}
    { X_i  \fforwardslash  \bar X_{i+1} &[-10] X_i  \fforwardslash  \left(\bar Y_{i+1} \bbackslash \bar Z_{i+1}\right) &[-10] Y_i  \fforwardslash  \bar Y_{i+1} & Z_i  \fforwardslash  \bar Z_{i+1} &[-5] Z_i \fforwardslash \bar Z_{i+1} \\
      \bar X_i & \bar X_i & \bar Y_i & \bar Z_i & \bar Z_i \\
      X_{i-1} \bbackslash \bar X_{i-1} & X_{i-1} \bbackslash \bar X_{i-1} & Y_{i-1} \bbackslash \bar Y_{i-1} & Z_{i-1} \bbackslash \bar Z_{i-1} & Z_{i-1} \bbackslash \left(\bar Y_{i-1}  \fforwardslash  \bar X_{i-1}\right) \\};
    \eto{1-2}{1-1} \mto{1-2}{1-3} \eto{1-4}{1-3} \eq{1-4}{1-5}
    \eq{2-2}{2-1} \mto{2-2}{2-3} \eto{2-4}{2-3} \eq{2-4}{2-5}
    \eq{3-1}{3-2} \mto{3-2}{3-3}  \eto{3-4}{3-3} \mto{3-4}{3-5}

    \eto{2-1}{1-1} \mto{2-1}{3-1}
    \eto{2-2}{1-2} \mto{2-2}{3-2}
    \eto{2-3}{1-3} \mto{2-3}{3-3}
    \eto{2-4}{1-4} \mto{2-4}{3-4}
    \eto{2-5}{1-5} \mto{2-5}{3-5}
    
    \comm{1-2}{2-3} \comm{2-3}{3-4} \comm{2-1}{3-2} \comm{1-4}{2-5}
    \end{diagram-numbered}
    Here, the vertical map $ X_i  \fforwardslash  \left(\bar Y_{i+1} \bbackslash \bar Z_{i+1}\right)\erto X_i  \fforwardslash  \bar X_{i+1}$ is induced by the horizontal map $\bar X_{i+1} \mrto \bar Y_{i+1} \bbackslash \bar Z_{i+1}$ in \eqref{eqn.chainLESdiagram} by using \cite[Lemma 2.10]{CZ-cgw} in a manner identical to that of the proof of Theorem \ref{thm:weaksnakelemma}. 
    
    This diagram is precisely of the form required in Theorem \ref{thm:snakelemma}, and by definition of homology (Definition \ref{defn:homology}) the resulting exact sequence is between the homologies of the chain complexes as in \eqref{eqn.LESpairproof}.
    \begin{diagram-numbered}{eqn.LESpairproof}
    { &[-15] \dotp &[-15]  &[-15] \dotp &[-15] &[-15] \dotp &[-20] &[-20] \dotp &[-20] &[-20] \dotp &[-20] \\
    H_i(X_\dotp) & & H_i(Y_\dotp) & & H_i(Y_\dotp) & & H_{i-1}(X_\dotp) & & H_{i-1}(Y_\dotp) & & H_{i-1}(Z_\dotp) \\};
     \eto{1-2}{2-1}\mto{1-2}{2-3} \eto{1-4}{2-3} 
    \mto{1-4}{2-5} \eto{1-6}{2-5} \mto{1-6}{2-7}
    \eto{1-8}{2-7} \mto{1-8}{2-9} \eto{1-10}{2-9} \mto{1-10}{2-11}
    \end{diagram-numbered}
    Piecing together these diagrams for all $i$, which can be straightforwardly checked to agree on the overlapping portions, gives the desired long exact sequence.
 \end{proof}

To further illustrate how the strong version of the Snake Lemma applies here in the case of $\FFinSet$, in \eqref{eqn.snakeforLES} is a picture of the diagram Theorem \ref{thm:snakelemma} was applied to in the proof, colored analogously to \eqref{eqn.strongsnake}.
\begin{equation}\label{eqn.snakeforLES}
\raisebox{-.5\height}{\includegraphics[height=3.5cm]{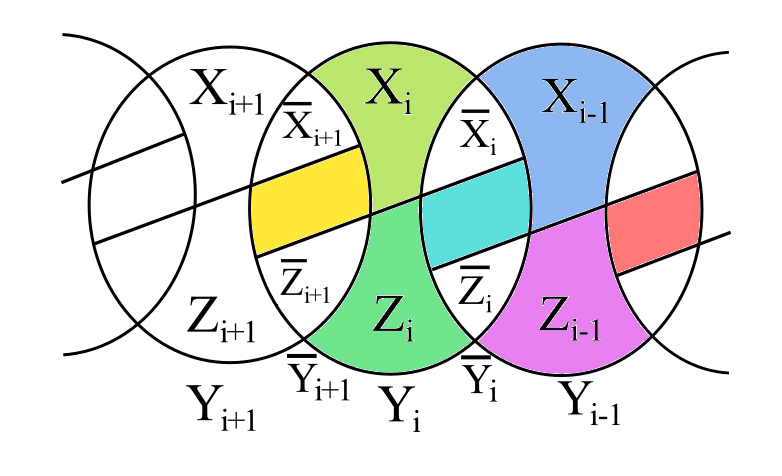}}
\end{equation}

\subsection{Homology as a functor}

In Section \ref{section:homologydefn}, we defined homology as an assignment that takes a chain complex $X_\dotp$ in an ACGW category $\bbA$ and produces an object $H_i(X_\dotp)$ of $\bbA$ for each $i$. In fact, this assignment extends to maps in a functorial way, as we show in this subsection. 

\begin{proposition}\label{prop:homologyonmaps}
    A morphism of chain complexes  $f_\dotp\colon X_\dotp\to Y_\dotp$ in $\Ch_\bbA$
    induces a morphism in $\bbA^\flat$ \[H_i(X_\dotp)\elto \dotp \mrto H_i(Y_\dotp)\]  
\end{proposition}

Before we formally construct this span, consider the picture we drew in Example \ref{chainmapsets} of a morphism of chain complexes in $\FFinSet$. On the left and right in \eqref{eqn.inducedspan} are $H_i(X_\dotp)$ and $H_i(Y_\dotp)$ shaded in that picture.
\begin{equation}\label{eqn.inducedspan}
\raisebox{-.45\height}{\includegraphics[height=3cm]{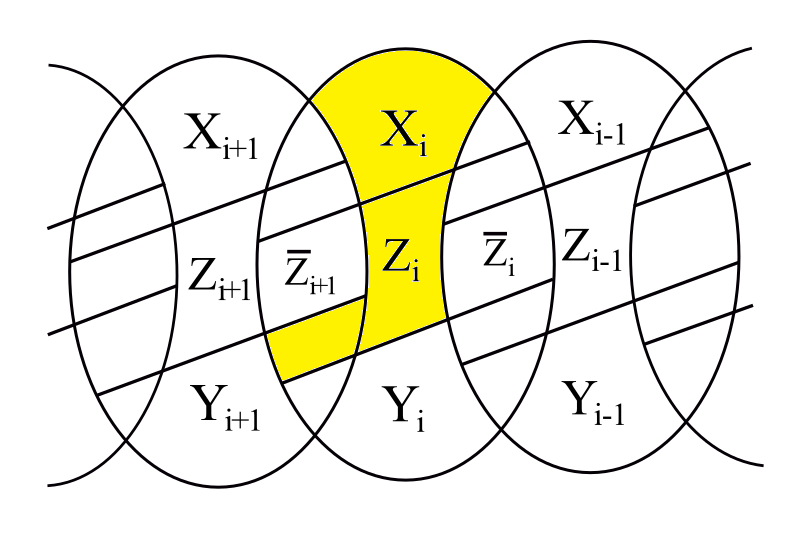}\qquad
\includegraphics[height=3cm]{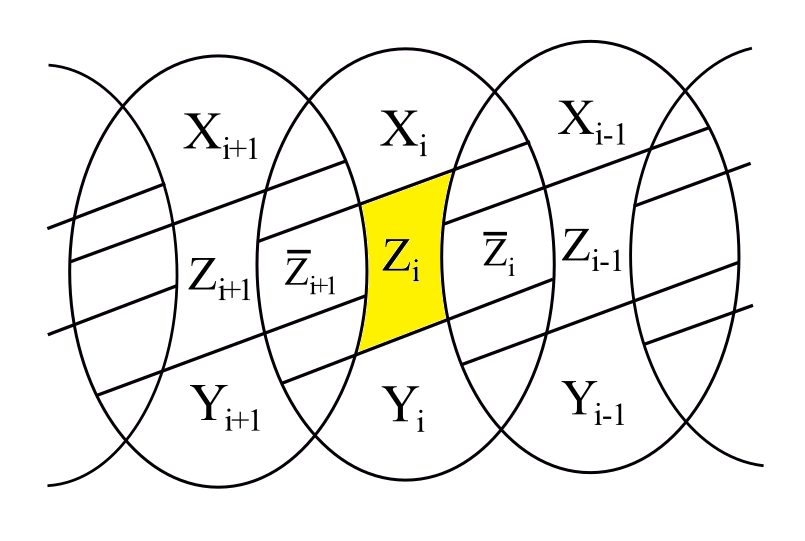}\qquad
\includegraphics[height=3cm]{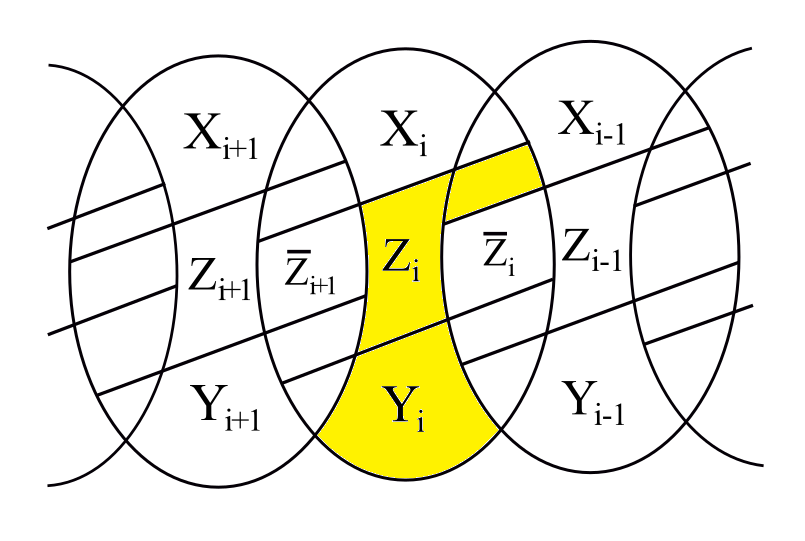}}
\end{equation}
Evidently from the pictures, these two ``intersect'' in the sense that $Z_i \backslash (\bar X_i \cup \bar Y_{i+1})$ includes into both and is the maximal subset of $Z_i$ to do so (in the center of \eqref{eqn.inducedspan}).\footnote{Due to the limitations of picture labeling, the center shaded region, while containing the label $Z_i$, does not make up all of $Z_i$ which includes all of the center oval between the two diagonal lines.} The proof then merely amounts to constructing this subset in the language of ACGW categories, which we can do using Theorem \ref{thm:lesforhomology}.

\begin{proof}
    A morphism of chain complexes of the form $X_\dotp \elto Z_\dotp \mrto Y_\dotp$ induces two short exact sequences of chain complexes,
    \[
    X_\dotp \bbackslash Z_\dotp \mrto X_\dotp \elto Z_\dotp 
    \qquad\qquad\textrm{and}\qquad\qquad 
    Z_\dotp \mrto Y_\dotp \elto Y_\dotp  \fforwardslash  Z_\dotp.
    \]
    By Theorem \ref{thm:lesforhomology}, these induce long exact sequences on homology which include, respectively, spans of the form
    \[
    H_i(X_\dotp) \elto \dotp \mrto H_i(Z_\dotp)
    \qquad\qquad\textrm{and}\qquad\qquad 
    H_i(Z_\dotp) \elto \dotp \mrto H_i(Y_\dotp).
    \]
    We can then compose these two spans in the category $\FFinSet^\flat$ (that is, take the mixed pullback of the two morphisms into $H_i(Z_\dotp)$) to get a span $H_i(X_\dotp) \elto \dotp \mrto H_i(Y_\dotp)$.
\end{proof}

\begin{example}
  It is not possible to restrict $H_i$ to a functor which lands in $\bbA$, even
  when restricting to horizontal (resp.\ vertical) morphisms in $\CCh_\bbA$.
 In other words, a horizontal chain morphism $X_\dotp\mrto Y_\dotp$ does
 \emph{not} necessarily induce a horizontal morphism $H_i(X_\dotp)\mrto
 H_i(Y_\dotp)$. For instance, if we consider the map between chain
 complexes in \eqref{eqn.spanex} of finite sets concentrated in degrees 1,2 and 3 
\begin{diagram-numbered}{eqn.spanex}
    {\emptyset & \emptyset  & \{a\} & \emptyset & \emptyset\\
    \{a\} & \{a\} & \{a,b\} & \{b\} & \{b\}\\};
    \eq{1-2}{1-1} \mto{1-2}{1-3} \eto{1-4}{1-3} \eq{1-4}{1-5}
    \eq{2-2}{2-1} \mto{2-2}{2-3} \eto{2-4}{2-3} \eq{2-4}{2-5}
    
    \mto{1-1}{2-1}
    \mto{1-2}{2-2}
    \mto{1-3}{2-3}
    \mto{1-4}{2-4} 
    \mto{1-5}{2-5} 

    \path[font=\scriptsize] (m-1-2) edge[-,white] node[pos=0.05]
  {\textcolor{black}{$\llcorner$}}  (m-2-1);
  \path[font=\scriptsize] (m-1-4) edge[-,white] node[pos=0.05]
  {\textcolor{black}{$\llcorner$}}  (m-2-3);
    \end{diagram-numbered}
    we see that $H_2(X_\dotp)=\{a\}$ while $H_2(Y_\dotp)=\emptyset$.  Thus there
    is not a horizontal morphism $H_2(X_\dotp) \mrto H_2(Y_\dotp)$, even though
    there is a morphism in $\Ch_{\FFinSet}^\flat$.
\end{example}

\begin{theorem}\label{thm:functorialhomology}
    Homology is a functor $\Ch_\bbA \rto \bbA^\flat$.
\end{theorem}
\begin{proof}
    We have already defined homology on objects in Definition \ref{defn:homology}, and on morphisms in Proposition \ref{prop:homologyonmaps}. A careful study of the proof of Proposition \ref{prop:homologyonmaps} reveals that if the morphism of chain complexes we start with is the identity, then it will be mapped to the identity span on homology (i.e. the equivalence class of the span whose legs are both identity maps). It remains to show that this assignment respects composition.

    Moreover as morphisms of chain complexes $X_\dotp \to Y_\dotp$ factor as composites of the form $X_\dotp \elto Z_\dotp \mrto Y_\dotp$, and the span associated to a morphism in Proposition \ref{prop:homologyonmaps} is defined as the corresponding composite of spans in $\bbA^\flat$, it suffices to check that composition is respected for composites of only horizontal morphisms. The argument for vertical morphisms will be entirely dual.

    Let $f: X_\dotp \mrto Y_\dotp$ and $g: Y_\dotp \mrto Z_\dotp$ be horizontal morphisms in $\CCh_\bbA$; we will show that $H_i(g) \circ H_i(f) = H_i(g\circ f)$.    Write $X_i' = X_i \fforwardslash \bar X_i$, so that $H_i(X) = X_i' \bbackslash \bar X_{i+1}$ (and similarly for $H_i(Y)$ and $H_i(Z)$).  To check that composition is respected, it suffices to check that it is respected when $H_i(f)$ is a single horizontal morphism, and $H_i(g)$ is a single vertical morphism (instead of a general span).  The former happens when $\bar X_{i+1} = \bar Y_{i+1}\times_{Y_i'} X_i'$; the latter happens when $Y_i = Z_i$ (and thus $Y_i' = Z_i'$).  We therefore focus on this case.  In this case, the data of the two compositions reduces to the diagram in \eqref{eqn.spanfunctor}, where $H_i(g) \circ H_i(f)$ is the composition around the lower-left of the rightmost square, and $H_i(g\circ f)$ is the span around the upper-right of the rightmost square.
    \begin{diagram-numbered}{eqn.spanfunctor}
        { \bar Y_{i+1} \times_{Y_i} X_i' & \bar Z_{i+1} \times_{Z_i'} X_i' & X_i ' & H_i(X) & X_i \bbackslash (\bar Z_{i+1} \times_{Z_i} X_i') \\
        \bar Y_{i+1} & \bar Z_{i+1} & Z_i' & H_i(Y) & H_i(Z) \\};
        \mto{1-1}{1-2} \mto{1-2}{1-3} \eto{1-4}{1-3} \eto{1-5}{1-4}
        \mto{2-1}{2-2} \mto{2-2}{2-3} \eto{2-4}{2-3} \eto{2-5}{2-4}
        \mto{1-1}{2-1} \mto{1-2}{2-2} \mto{1-3}{2-3} \mto{1-4}{2-4} \mto{1-5}{2-5}
        \pbsym{1-2}{2-3} \comm{1-3}{2-4}
    \end{diagram-numbered}
    Thus it suffices to check that the rightmost square is pseudo-commutative.  
    The composition of the two squares on the left is a pullback square, and the composition of the two squares on the right is a pseudo-commutative square.  The statement that $H_i$ respects composition is then the statement that this implies that the rightmost square in the diagram is a pseudo-commutative square.  This is part of the naturality of the duality between pullback squares as on the left and pseudo-commutative squares as on the right;  on the left, the leftmost square must be a pullback square by a standard argument. 
\end{proof}

\subsection{Quasi-isomorphisms}\label{subsection:quasiiso}

Just like in the setting of abelian categories, we can use our notion of homology to define what it means for a morphism of chain complexes of sets to be a quasi-isomorphism. 

\begin{definition}
    A morphism $f_\dotp\colon X_\dotp\to Y_\dotp$ of chain complexes in an ACGW category $\bbA$ is a \emph{quasi-isomorphism} if it induces an isomorphism on homology; that is, if the associated span on homology is of the form \[H_i(X_\dotp)\elto^{\cong} \dotp \mrto^{\cong} H_i(Y_\dotp)\] for all $i$.
\end{definition}

\begin{example}
    If $\bbA$ arises from an abelian category, this agrees with the usual definition that the induced morphism on homology $H_i(X_\dotp) \to H_i(Y_\dotp)$ is an isomorphism for all $i$.
\end{example}

\begin{example}
    In $\FFinSet$, a quasi-isomorphism can be pictured as a morphism of chain complexes in which all three of the highlighted subsets in \eqref{eqn.inducedspan} agree. Unwinding this, the condition becomes that $X_i \backslash \left(\bar X_i \cup \bar X_{i+1} \cup Z_i\right)$, $Y_i \backslash \left(\bar Y_i \cup \bar Y_{i+1} \cup Z_i\right)$, $\bar Y_{i+1} \cap Z_i \backslash \bar Z_{i+1}$, and $\bar X_i \cap Z_i \backslash \bar Z_i$ (those subsets highlighted on the left in \eqref{eqn.setquasiiso}) are all empty. 
    \begin{equation}\label{eqn.setquasiiso}
    \raisebox{-.5\height}{\includegraphics[height=3.5cm]{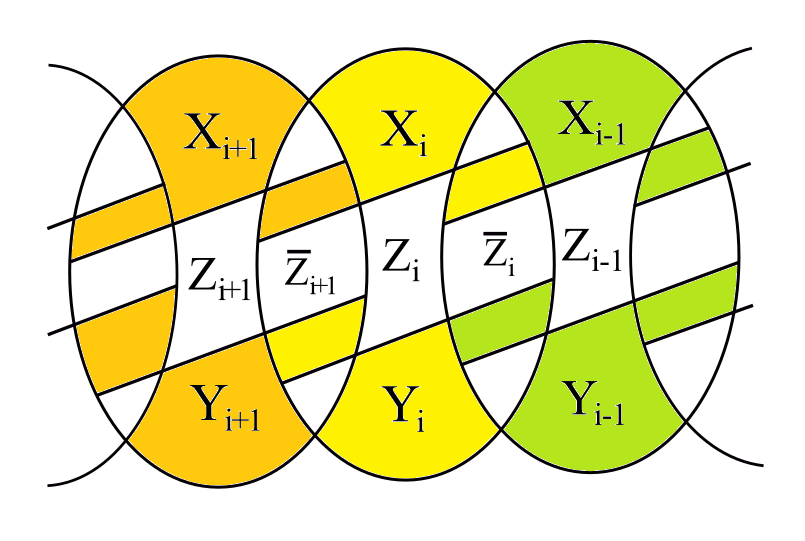}}
    \qquad\qquad
    \raisebox{-.5\height}{\includegraphics[height=3cm]{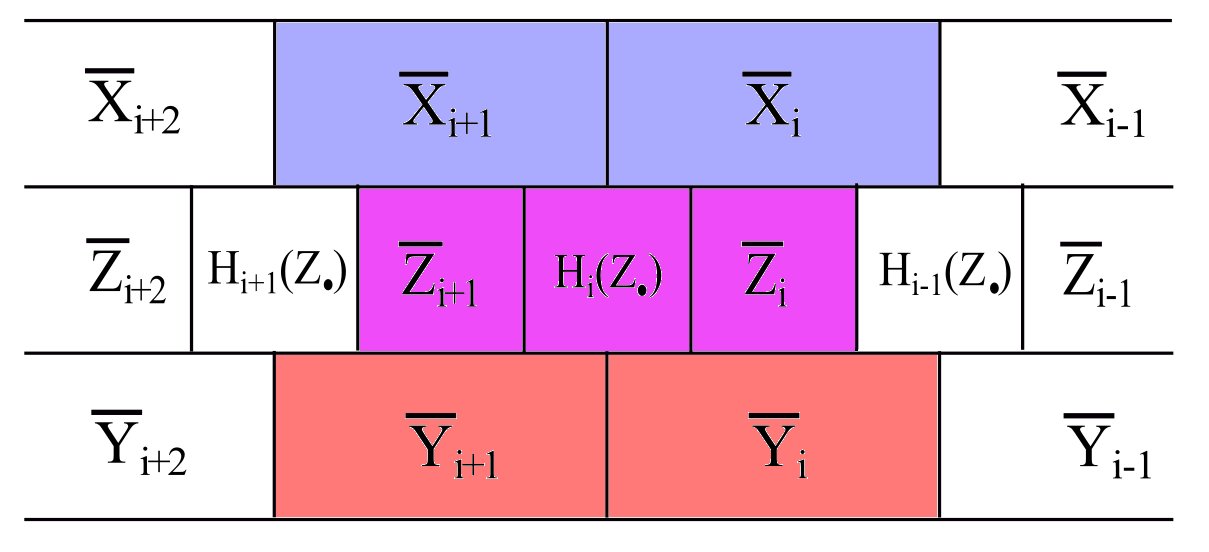}}
    \end{equation}
    
    A quasi-isomorphism can therefore be pictured as above right, where $X_i$ is shaded blue, $Y_i$ is shaded red, and $Z_i$ which includes into both is shaded purple.
\end{example}

Given  a morphism of chain complexes $f_\dotp \colon X_\dotp\to Y_\dotp$  in an abelian category, in some cases it is possible to characterize whether $f_\dotp$ is a quasi-isomorphism without comparing the homologies of $X_\dotp$ and  $Y_\dotp$. Namely, if we know that $f_\dotp$ is a monomorphism (respectively, epimorphism), then it is a quasi-isomorphism if and only if it's cokernel (resp.\ kernel) is an exact chain complex. This characterization generalizes to any ACGW category.

\begin{proposition}
    A horizontal morphism of chain complexes $f_\dotp \colon X_\dotp\mrto Y_\dotp$ is a quasi-isomorphism if and only if its cokernel is an exact chain complex. Likewise a vertical morphism of chain complexes is a quasi-isomorphism if and only if its kernel is exact.
\end{proposition}
\begin{proof}
    If we consider the short exact sequence of chain complexes \[X_\dotp \mrto^{f_\dotp} Y_\dotp \elto Y_\dotp\fforwardslash X_\dotp\] then by Theorem \ref{thm:lesforhomology} we get an induced long exact sequence as in \eqref{eqn.LESpairqproof}
    \begin{diagram-numbered}{eqn.LESpairqproof}
    { &[-20] \dotp &  & \dotp & & \dotp &[-15]   \\
    \cdots H_{i+1}(Y_\dotp \fforwardslash X_\dotp)  & & H_i(X_\dotp) & & H_i(Y_\dotp) & & H_i(Y_\dotp \fforwardslash X_\dotp) \cdots\\};
     \eto{1-2}{2-1}\mto{1-2}{2-3} \eto{1-4}{2-3} 
    \mto{1-4}{2-5} \eto{1-6}{2-5} \mto{1-6}{2-7}
    \end{diagram-numbered}

    Suppose that the chain complex $Y_\dotp\fforwardslash X_\dotp$ is exact. Then the long exact sequence above must be of the form in \eqref{eqn.ifexact}
    \begin{diagram-numbered}{eqn.ifexact}
    { & \varnothing &  & \dotp & & \varnothing & &   \\
    \cdots \varnothing  & & H_i(X_\dotp) & & H_i(Y_\dotp) & & \varnothing \cdots\\};
     \eq{1-2}{2-1}\mto{1-2}{2-3} \eto{1-4}{2-3} 
    \mto{1-4}{2-5} \eto{1-6}{2-5} \eq{1-6}{2-7}
    \end{diagram-numbered} 
    as any morphism into $\varnothing$ is an isomorphism. 
    
    As this sequence is exact, the vertical morphism $H_i(X_\dotp) \elto \dotp$ must be an isomorphism as it is isomorphic to the cokernel of $\varnothing \mrto H_i(X_\dotp)$. Dually, the horizontal morphism $\dotp \mrto H_i(Y_\dotp)$ is an isomorphism as it is isomorphic to the kernel of $H_i(Y_\dotp) \elto \varnothing$. This shows that $f_\dotp$ is a quasi-isomorphism. 
    

    Conversely, suppose that $f_\dotp$ is a quasi-isomorphism. Then the long exact sequence is of the form in \eqref{eqn.ifqiso}. 
    \begin{diagram-numbered}{eqn.ifqiso}
    { &[-15] A_i &  & \dotp & & B_i &[-10] &[-10] A_{i-1}  \\
      H_{i+1}(Y_\dotp \fforwardslash X_\dotp)  & & H_i(X_\dotp) & & H_i(Y_\dotp) & & H_i(Y_\dotp \fforwardslash X_\dotp) \\};
     \eto{1-2}{2-1}\mto{1-2}{2-3} \eto{1-4}{2-3}_{\cong}
    \mto{1-4}{2-5}^{\cong} \eto{1-6}{2-5} \mto{1-6}{2-7} \eto{1-8}{2-7}_{\cong}
    \end{diagram-numbered}
    The previous argument then works in reverse, where as the kernel (resp. cokernel) of an isomorphism the object $A_i$ (resp. $B_i$) must be isomorphic to $\varnothing$ for all $i$. We then have the short exact sequence $\varnothing \mrto H_i(Y_\dotp \fforwardslash X_\dotp) \elto \varnothing$, which by the same principle implies that both maps are isomorphisms and hence $H_i(Y_\dotp \fforwardslash X_\dotp) \cong \varnothing$. This shows that $Y_\dotp \fforwardslash X_\dotp$ is an exact chain complex.

    The argument for a vertical morphism of chain complexes is entirely dual to this.
\end{proof}
    

\begin{remark}
    If $f_\dotp\colon X_\dotp\to Y_\dotp$ is any morphism of chain complexes, which we can express as a span \[X_\dotp \elto^{f_\dotp^1} Z_\dotp \mrto^{f_\dotp^2} Y_\dotp,\] then the fact that $f_\dotp$ is a quasi-isomorphism does \emph{not} imply that both $f_\dotp^1$ and $f_\dotp^2$ are also quasi-isomorphisms. For instance, a counter-example is given by the span of chain complexes in \eqref{eqn.spanexq} concentrated in degrees 1, 2 and 3.
    \begin{diagram-numbered}{eqn.spanexq}
    { \{a\} & \{a\} & \{a,b\} & \{b\} & \{b\}\\
    \emptyset & \emptyset  & \{b\} & \emptyset & \emptyset\\
    \{b\} & \{b\} & \{b\} & \emptyset & \emptyset\\};
    \eq{1-2}{1-1} \mto{1-2}{1-3} \eto{1-4}{1-3} \eq{1-4}{1-5}
    \eq{2-2}{2-1} \mto{2-2}{2-3} \eto{2-4}{2-3} \eq{2-4}{2-5}
    \eq{3-2}{3-1} \eq{3-2}{3-3} \eto{3-4}{3-3} \eq{3-4}{3-5}
    
    \eto{2-1}{1-1}
    \eto{2-2}{1-2}
    \eto{2-3}{1-3}
    \eto{2-4}{1-4}
    \eto{2-5}{1-5}
    \mto{2-1}{3-1} \mto{2-2}{3-2} \mto{2-3}{3-3} \eq{2-4}{3-4} \eq{2-5}{3-5}

    \path[font=\scriptsize] (m-2-2) edge[-,white] node[pos=0.05]
  {\textcolor{black}{$\urcorner$}}  (m-1-3);
  \path[font=\scriptsize] (m-2-4) edge[-,white] node[pos=0.05]
  {\textcolor{black}{$\urcorner$}}  (m-1-5);
  \path[font=\scriptsize] (m-2-2) edge[-,white] node[pos=0.05]
  {\textcolor{black}{$\llcorner$}}  (m-3-1);
  \path[font=\scriptsize] (m-2-4) edge[-,white] node[pos=0.05]
  {\textcolor{black}{$\llcorner$}}  (m-3-2);
    \end{diagram-numbered}
\end{remark}

\begin{remark}
  In the case of finite sets, homology behaves more like chain complexes over a
  field, rather than chain complexes over a ring.  The most striking example of
  this phenomenon is that, over a field, every chain complex has a quasi-isomorphism
  to and from a chain complex in which all differentials are zero.  (This is sometimes
  stated as ``the chain complex is quasi-isomorphic to its homology.'')  This
  can be done quite simply: for each vector space in the complex, choose a basis
  in such a way that it restricts to a basis of the kernel of the next
  differential, and to the image of the previous.  Then a subset of the basis
  will give a basis for the homology, as well.  Note that there are two possible
  ways of constructing a quasi-isomorphism to the homology: one by mapping the
  homology into the chain complex via this basis, and one quotienting out by the
  basis elements that do not generate the homology.  (Over a ring it is generally not the case that this is possible; for a simple example consider the chain complex over $\mathbf{Z}$ which contains the single map $2: \mathbf{Z}/4 \rto \mathbf{Z}/4$.)

  In the case of sets this is done similarly, without the non-canonical step of
  choosing the basis.  The homology of the chain complex of sets includes
  naturally into the original chain complex, giving a quasi-isomorphism.  The
  choice of whether to make this inclusion vertical or horizontal is
  non-canonical, but it corresponds naturally to the choice of inclusion or
  quotient in the field case.  
\end{remark}

\bibliographystyle{alpha}
\bibliography{SSZ}

\end{document}